\documentclass[11pt]{amsart}

\usepackage{amsfonts,amssymb,amscd,amstext}
\IfFileExists{mathrsfs.sty}{\usepackage{mathrsfs}}{\let\mathscr\mathcal}
\usepackage[utf8]{inputenc}
\usepackage{hyperref}
\usepackage{verbatim}

\IfFileExists{ptmr7t.tfm}{\usepackage{times}}{}
\IfFileExists{enumerate.sty}{\usepackage{enumerate}}{%
  \let\originalenumerate\enumerate
  \let\endoriginalenumerate\endenumerate
  \renewenvironment{enumerate}[1][]{\originalenumerate}{\endoriginalenumerate}%
}

\numberwithin{equation}{section}

\newtheorem{theorem}{Theorem}[section] 
\newtheorem{proposition}[theorem]{Proposition} 
 
\newtheorem{lemma}[theorem]{Lemma} 
\newtheorem{corollary}[theorem]{Corollary} 

\theoremstyle{definition} 
\newtheorem{definition}[theorem]{Definition} 
\newtheorem{remark}[theorem]{Remark}

\newtheorem{example}[theorem]{Example} 
 
\newcommand\Acal{\mathcal{A}} 
\newcommand\Bcal{\mathcal{B}} 
\newcommand\Ccal{\mathcal{C}}

\newcommand\Ucal{\mathcal{U}}

\newcommand\Ascr{\mathscr{A}} 
 
\newcommand\Cscr{\mathscr{C}}

\newcommand\Oscr{\mathscr{O}}

\newcommand\C{\mathbb{C}} 
\newcommand\D{\overline{\mathbb D}} 
 
\renewcommand\D{\mathbb D}

\newcommand\N{\mathbb{N}} 
 
\newcommand\R{\mathbb{R}}

\newcommand\U{\mathbb{U}} 
\newcommand\Z{\mathbb{Z}}

\newcommand\igot{\mathfrak{i}}

\renewcommand\igot{\mathfrak{i}}

\renewcommand\imath{\igot}

\newcommand\hra{\hookrightarrow} 
\newcommand\lra{\longrightarrow}

\newcommand\wt{\widetilde} 
 
\newcommand\di{\partial} 
\newcommand\dibar{\overline\partial}

\newcommand\Id{\mathrm{Id}} 
\newcommand\pr{\mathrm{pr}} 
\newcommand\rank{\mathrm{rank}} 

\def\Ell1{\mathrm{Ell_1}} 
\def\CEll1{\mathrm{CEll_1}}

\newcommand\supp{\mathrm{supp}}

\newcommand\Aut{\mathrm{Aut}}

\newcommand\Lin{\mathrm{Lin}}
\newcommand\Lip{\mathrm{Lip}}

\makeatletter
\@namedef{subjclassname@2020}{\textup{2020} Mathematics Subject Classification}
\makeatother

\begin{document} 

\title{The Oka principle for holomorphic fibre bundles of 
H\"older--Zygmund classes on strongly pseudoconvex domains}

\author{Franc Forstneri{\v c}}

\address{Franc Forstneri{\v c}, Faculty of Mathematics and Physics, University of Ljubljana, Jadranska 19, 1000 Ljubljana, Slovenia}

\address{Franc Forstneri{\v c}, Institute of Mathematics, Physics, and Mechanics, Jadranska 19, 1000 Ljubljana, Slovenia}

\email{franc.forstneric@fmf.uni-lj.si}



\date{24 May 2026. Revised on 28 November 2026}


\begin{abstract} 
Let $\overline \Omega$ be a compact strongly pseudoconvex domain
with $\Cscr^2$ boundary in a Stein manifold, 
and let $h:Z\to \overline \Omega$ be a 
fibre bundle of little H\"older--Zygmund class \(\lambda^r\), \(r>0\), 
which is holomorphic over \(\Omega\).
Assuming that the fibre is an Oka manifold, we prove that 
every continuous section \(f_0:\overline \Omega\to Z\) 
is homotopic to a section \(f_1:\overline \Omega\to Z\) 
of class \(\lambda^r(\overline \Omega)\) which is holomorphic on \(\Omega\).
We also establish the parametric h-principle in this context.  
As an application, we obtain the Oka principle for the classification
of vector bundles and principal bundles of little 
H\"older--Zygmund classes on such domains. 
For a fibre bundle of big H\"older--Zygmund class $\Lambda^r_\Oscr$ with Oka fibre, every continuous section is homotopic to a section of class 
$\lambda^s_\Oscr$ for each $0<s<r$.
\end{abstract}

\maketitle

\centerline{\em On the Centenary of the Birth of Professor Lu Qikeng}

\setcounter{tocdepth}{1}
\tableofcontents

%
%
\section{Introduction}
\label{sec:intro}

A complex manifold $Y$ is said to be an {\em Oka manifold} 
(see \cite{Forstneric2009CR} and \cite[Chap.\ 5]{Forstneric2017E}) if maps
$X\to Y$ from any Stein manifold $X$ satisfy the h-principle, also called
the Oka principle. This means in particular that any continuous map 
$f_0:X\to Y$ is homotopic to a holomorphic map $f_1:X\to Y$;
if $f_0$ is holomorphic on a neighbourhood of a compact holomorphically
convex subset $K$ of $X$ and on a closed complex subvariety 
$X'$ of $X$, then a homotopy $f_t:X\to Y$ $(t\in [0,1])$ 
from $f_0$ to a holomorphic map $f_1$ can be chosen to 
be fixed on $X'$ and to consist of maps which are holomorphic on a 
neighbourhood of $K$ and uniformly close to $f_0$ on $K$. 
The analogous results hold for sections of holomorphic fibre bundles with 
Oka fibres, and for sections of elliptic holomorphic submersions $Z\to X$ 
onto a Stein space. For the theory of Oka manifolds and Oka maps, see 
\cite{Forstneric2023Indag,Forstneric2026ICM,Forstneric2017E,ForstnericLarusson2011}. 
Classical examples of Oka manifolds include complex homogeneous 
manifolds (see Grauert \cite{Grauert1957I,Grauert1957II} and 
\cite[Proposition 5.6.1]{Forstneric2017E}) and, more generally, 
Gromov elliptic manifolds 
(see \cite{Gromov1989} and \cite[Corollary 5.6.14]{Forstneric2017E}).
Oka manifolds $Y$ are characterised by the following approximation property: 
every holomorphic map from a neighbourhood of a compact convex set $K\subset\C^n$ to $Y$ can be approximated uniformly on $K$ by entire maps 
$\C^n\to Y$ (the {\em convex approximation property}, CAP;
see \cite{Forstneric2009CR} and \cite[Theorem 5.4.4]{Forstneric2017E}). 
They are also characterised by {\em convex relative ellipticity}, CRE  
(see Kusakabe \cite[Theorem 1.3]{Kusakabe2021IUMJ} and
\cite[Definition 1.5 and Theorem 1.6]{Forstneric2026ICM}).
In the recent preprint \cite{DGWX2026} by An, Du, Guo, Wang, and Xie,
Oka manifolds are characterised by strong dominability by affine
spaces; this has already led to many newly discovered examples. 

Let $\Omega$ be a relatively compact Lipschitz domain 
in a smooth Riemannian manifold $X$. Given a number $r>0$, we 
have the little and the big Zygmund spaces
$\lambda^r(\overline\Omega)\subset \Lambda^r(\overline\Omega)$, 
where $\lambda^r(\overline\Omega)$ is the closure of smooth functions 
on $\overline\Omega$ in the $\Lambda^r(\overline\Omega)$ norm.
If $r=k+\alpha$ with $k\in \Z_+=\{0,1,2,\ldots\}$ and $0<\alpha<1$ then
$\Lambda^r(\overline\Omega)$ coincides with the H\"older space
$\Cscr^{k,\alpha}(\overline\Omega)$.
For a domain $\Omega$ in a complex manifold, $\Oscr(\Omega)$ 
denotes the algebra of holomorphic functions on $\Omega$, and 
$\lambda^r_\Oscr(\overline\Omega)\subset \Lambda_\Oscr^r(\overline\Omega)$
denote the little and the big Zygmund spaces of functions 
which are holomorphic on $\Omega$. The precise definition of 
Zygmund spaces $\Lambda^r(\overline\Omega)$ is given in 
Subsect.\ \ref{ss:HZspaces}, the little Zygmund spaces 
$\lambda^r(\overline\Omega)$ 
are defined in Subsect.\ \ref{ss:little}, manifold-valued maps
of Zygmund classes are introduced in Subsect.\ \ref{ss:manifold-valued}, 
and fibre bundles are introduced in Subsect.\ \ref{ss:Lr-bundles}.

Our first main result is the following 1-parametric
Oka principle for sections of fibre bundles of little Zygmund classes 
$\lambda^r_\Oscr(\overline\Omega)$ with Oka fibres on compact, 
strongly pseudoconvex domains $\overline \Omega$ with $\Cscr^2$ 
boundary in Stein manifolds. 

%
%
\begin{theorem}[Little Zygmund space Oka principle]
\label{th:main}
Assume that $\Omega$ is a relatively compact, strongly pseudoconvex 
domain with $\Cscr^2$ boundary $b\Omega$ in a Stein manifold $X$, $r>0$
is a real number, and $h:Z\to \overline \Omega$ is a fibre bundle of little Zygmund 
class $\lambda^r(\overline \Omega)$ which is holomorphic on $\Omega$
and whose fibre is an Oka manifold. Then, every continuous
section $f_0\in \Gamma(\overline \Omega,Z)$ of $h$ is homotopic to a section 
$f\in \Gamma^r_\Oscr(\overline \Omega,Z)$ of class 
$\lambda^r(\overline \Omega)$ which is holomorphic on $\Omega$. 
Furthermore, every continuous path 
$ 
	[0,1]\ni t\mapsto f_t\in\Gamma(\overline\Omega,Z) 
$ 
whose endpoints belong to $\Gamma^r_\Oscr(\overline\Omega,Z)$ can be deformed, with fixed endpoints, to a path in 
$\Gamma^r_\Oscr(\overline\Omega,Z)$ which is continuous in the topology induced by the local $\Lambda^r$ norms.
\end{theorem}

Theorem \ref{th:main} is proved in Sect.\ \ref{sec:proofmain}.
There we also give a fully parametric version, Theorem \ref{th:mainpar}.
The proof yields an approximation result: if the section $f_0$ is holomorphic
on a neighbourhood of a compact holomorphically convex subset
$K\subset \Omega$, then the homotopy from $f_0$ to a holomorphic
section $f\in \Gamma^r_\Oscr(\overline \Omega,Z)$ can be chosen to consist 
of sections that are holomorphic near $K$ and approximate $f_0$ uniformly 
on $K$. 

Theorem \ref{th:main} implies the following Oka principle for sections 
of fibre bundles of the big Zygmund class $\Lambda^r_\Oscr(\overline\Omega)$,
with an arbitrarily small decrease of the Zygmund exponent.

%
%
\begin{theorem}[Big Zygmund space Oka principle with loss of exponent]
\label{th:main-big}
Let $\Omega$ be a relatively compact, strongly pseudoconvex domain with 
$\Cscr^2$ boundary in a Stein manifold $X$, and let
$h:Z\to\overline\Omega$ be a fibre bundle of class
$\Lambda^r_\Oscr(\overline\Omega)$, $r>0$, with Oka fibre.
Given a number $s\in (0,r)$, every continuous section of $h$
is homotopic to a section of little class $\lambda^{s}_\Oscr(\overline\Omega)$, 
and hence of big class $\Lambda^{s}_\Oscr(\overline\Omega)$.
Furthermore, every homotopy of continuous sections whose two
end sections are of class $\Lambda^r_\Oscr$ can be deformed,
with fixed ends, to a homotopy continuous in the
$\Lambda^{s}_\Oscr$ topology and consisting of little-class
$\lambda^{s}_\Oscr$ sections.
\end{theorem}

\begin{proof}
For $0<s<r$ the continuous inclusion 
$\Lambda^r\hookrightarrow\lambda^{s}$ 
of Lemma \ref{lem:interpolation-dilation} applies to transition maps in local 
bundle trivialisations, so the bundle is of class $\lambda^{s}_\Oscr$. 
Apply Theorem \ref{th:main} at exponent $s$. 
\end{proof}

The relative parametric version of Theorem~\ref{th:main-big} is Theorem~\ref{th:mainpar-big}.

\begin{remark}
We do not know whether Theorem \ref{th:main} holds in big Zygmund spaces
$\Lambda^r_\Oscr(\overline\Omega)$.
The reason is that the approximation result in Theorem \ref{th:approximation},
which is used in the proof, fails for the same Zygmund exponent 
in big spaces; see Example \ref{ex:density-obstruction}.
\end{remark}

The analogues of Theorem \ref{th:main}, and of several other 
results of the paper, in the spaces
\[
	 \Ascr^r(\overline\Omega)=
	 \{f\in\Cscr^r(\overline\Omega):f|_\Omega\in\Oscr(\Omega)\},
	 \qquad r\in\Z_+=\{0,1,2,\ldots\}
\]
were obtained by Drinovec Drnov\v sek and Forstneri\v c  \cite{DrinovecForstneric2008FM}. 
One motivation for the present generalisation is that the Zygmund 
spaces behave better than the $\Cscr^k$ or Lipschitz spaces 
for many natural operators considered in analysis, as was already 
noticed by Zygmund in 1945. A more explicit motivation was a question 
by Andrei Teleman (private communication, January 2026) 
whether a result in the spirit of Theorem \ref{th:main} holds. After the 
completion of this work, he obtained a related result on extending 
principal holomorphic fibre bundles of Zygmund classes \cite{Teleman2026} 
on strongly pseudoconvex extensions.

We work with compact strongly pseudoconvex domains 
$\overline \Omega$ in a Stein manifold. In view of the embedding theorems 
\cite{Ohsawa1984LNM,Heunemann1986III,Catlin1988}, our
results hold for any compact complex manifold with Stein interior and 
$\Cscr^\infty$ smooth strongly pseudoconvex boundary.

We refer to Sect.\ \ref{sec:prelim} for the definition and properties of Zygmund spaces.
For $r=k+\alpha$ with $k\in\Z_+$ and $0<\alpha<1$, 
$\Lambda^r(\overline \Omega)$ coincides with the  
H\"older space $\Cscr^{k,\alpha}(\overline \Omega)$. 
For integer values $r=k+1\in \{1,2,\ldots\}$, 
$\Lambda^{k+1}(\overline \Omega)$ properly contains 
the Lipschitz space $\Cscr^{k,1}(\overline \Omega)$ 
of functions whose derivatives up to order $k$ 
are Lipschitz continuous on $\overline \Omega$
(see \cite[Example 4.3.1, p.\ 148]{SteinEM1970}). 
There are continuous strict embeddings among the 
classical and the Zygmund scales, 
see \eqref{eq:scale} and Lemma \ref{lem:interpolation-dilation}. 
The analytic input is a bounded linear solution operator for $\dibar$ that gains 
half a derivative in the big Zygmund spaces; see Theorem~\ref{th:dibar}. 
On the little Zygmund spaces, it gives a gain of any order strictly less 
than one half; see Corollary~\ref{cor:dibar-little}.

On the way to Theorem \ref{th:main}, we obtain several results of 
independent interest. Given complex manifolds $X$ and $Y$ and a 
Lipschitz domain $\Omega\Subset X$, set
\[ 
	\Lambda^r_\Oscr(\overline\Omega,Y) = 
	\{f\in\Lambda^r(\overline\Omega,Y): f|_\Omega\in\Oscr(\Omega,Y)\}, 
\] 
\[ 
	\lambda^r_\Oscr(\overline\Omega,Y) = 
	\{f\in\lambda^r(\overline\Omega,Y): f|_\Omega\in\Oscr(\Omega,Y)\}. 
\]
For a compact subset $K\subset X$, $\Oscr(K,Y)$ denotes the space of 
germs of maps to $Y$ holomorphic on open neighbourhoods of $K$. 

%
%
\begin{theorem}\label{th:approximation}
Let $\Omega$ be a relatively compact, strongly pseudoconvex domain with $\Cscr^2$ boundary in a Stein manifold $X$, let $Y$ be a complex manifold, and let $0<s<r$. Then: 
\begin{enumerate}[\rm (i)] 
\item Every $f\in\Lambda^r_{\Oscr}(\overline\Omega,Y)$ 
can be approximated by maps in $\Oscr(\overline\Omega,Y)$ in the 
$\Lambda^{s}(\overline\Omega,Y)$ topology. 
\item A map $f\in\Lambda^r_{\Oscr}(\overline\Omega,Y)$ admits such approximation in the $\Lambda^r$ topology if and only if $f\in\lambda^r_{\Oscr}(\overline\Omega,Y)$. 
\end{enumerate}
\end{theorem}

Example \ref{ex:density-obstruction} illustrates the failure of approximability
of a general function in $\Lambda^r(\overline\Omega)$ in the 
$\Lambda^r$ topology by functions holomorphic on neighbourhoods 
of $\overline\Omega$, even on the unit disc in $\C$. 

Theorem \ref{th:approximation} holds for every target manifold $Y$. For an Oka 
target, we have the following corollary which follows from 
Theorem \ref{th:approximation} and the Oka principle 
in \cite[Corollary 5.4.5]{Forstneric2017E}.

%
%
\begin{corollary}\label{cor:approximation}
Assume that $\Omega\Subset X$ are as in Theorem \ref{th:approximation},
$\overline\Omega$ is $\Oscr(X)$-convex, and $Y$ is an Oka manifold. 
Fix numbers $0<s<r$. Then, every continuous map $f:X\to Y$ such that 
$f|_{\overline\Omega} \in\Lambda^r_\Oscr(\overline\Omega,Y)$ can be approximated in $\Lambda^{s}(\overline\Omega)$ by holomorphic maps
$X\to Y$. Such approximation is possible in the 
$\Lambda^{r}(\overline\Omega)$ topology if and only if 
$f|_{\overline\Omega} \in\lambda^r_\Oscr(\overline\Omega,Y)$.
\end{corollary}

\begin{proof}
Theorem~\ref{th:approximation} gives a holomorphic map $g:U\to Y$ 
in a neighbourhood $U$ of $\overline\Omega$ approximating 
$f|_{\overline\Omega}$ in $\Lambda^s(\overline\Omega)$.
We can modify $f$ to a continuous map on $X$ that equals $g$ 
on a neighbourhood of $\overline\Omega$. Apply Oka approximation 
(see \cite[Theorem 5.4.4]{Forstneric2017E})
on a smaller $\Oscr(X)$-convex compact neighbourhood 
of $\overline\Omega$. Cauchy estimates on $\overline\Omega$ 
give the desired norm approximation. Necessity at exponent $r$ 
follows from Theorem~\ref{th:approximation}(ii).
\end{proof}

Theorem \ref{th:approximation} is proved in Sect.\ \ref{sec:approximation};
see also the parametric version in Theorem \ref{th:approximation-par}. 
The analogue of Theorem \ref{th:approximation} is known for spaces 
$\Ascr^r(\overline\Omega)$ with $r\in\Z_+$; 
see \cite[Theorem 2.9.2,\ p.\ 87]{HenkinLeiterer1984} 
and \cite[Theorem 24, p.\ 165]{FornaessForstnericWold2020}.  
For approximation of manifold-valued maps in 
$\Ascr^r(\overline\Omega,Y)$, where $Y$ is a complex manifold
and $r\in\Z_+$,  
by holomorphic maps from open neighbourhoods of $\overline\Omega$
in $X$, see \cite[Theorem 1.2]{DrinovecForstneric2008FM},
\cite[Theorem 8.11.4]{Forstneric2017E}, and 
\cite[Corollary 9, p.\ 178]{FornaessForstnericWold2020}.
The last mentioned result in \cite{FornaessForstnericWold2020}
is a more general Mergelyan-type approximation theorem 
on strongly admissible sets in Stein manifolds.

We also have the following result generalising 
\cite[Theorem 1.1 (i)]{Forstneric2007AJM}. The proof given in 
\cite[Sect.\ 2]{Forstneric2007AJM} also applies to Zygmund spaces.

\begin{theorem}\label{th:manifold}
Assume that $\Omega$ is a relatively compact, strongly pseudoconvex 
domain with $\Cscr^2$ boundary in a Stein manifold and $Y$ is a complex 
manifold. For every $r>0$ the space $\Lambda^r_\Oscr(\overline \Omega,Y)$ 
is a complex Banach manifold.
The tangent space $T_f \Lambda^r_\Oscr(\overline \Omega,Y)$ 
at a point $f\in \Lambda^r_\Oscr(\overline \Omega,Y)$
is the space of sections of class $\Lambda^r_\Oscr(\overline \Omega)$
of the complex vector bundle $f^*TY\to \overline \Omega$.
\end{theorem}

The analogous assertion holds for $\lambda^r_\Oscr(\overline \Omega,Y)$, 
with little-space tangent sections. The graph-neighbourhood charts
and holomorphic superposition operators defining these Banach
manifolds do not require density of smooth maps in the Zygmund norm.
Restricting the charts to sections of a little Zygmund space gives the 
little-space assertion.

In the proof of Theorem \ref{th:main}, we shall need several
results concerning vector bundles $E\to\overline\Omega$ of
class $\lambda^r_\Oscr(\overline\Omega)$ 
on strongly pseudoconvex Stein domains $\Omega$ 
with $\Cscr^2$ boundary; see Sect.\ \ref{sec:vectorbundles}. 
In particular, we obtain Theorem A for such bundles; 
see Theorems \ref{th:TheoremA} and \ref{th:TheoremApar}. 
The proof is based on the Cartan splitting lemma 
for maps of little Zygmund classes to a complex Lie group;
see Lemmas \ref{lem:Cartan}, \ref{lem:Cartan-par}
and Remark \ref{rem:Liegroup}. 
The main technical results used in the proof of  Theorem \ref{th:main}
are a splitting lemma (see Lemma \ref{lem:splitting}) 
and a gluing lemma (see Lemma \ref{lem:gluing})
for sprays of sections of class $\lambda^r_\Oscr$. 

An application of Theorem \ref{th:main} is 
the following Oka principle for holomorphic vector bundles of class 
$\lambda^r_\Oscr(\overline\Omega)$, proved in Sect.\ \ref{sec:OPVB}. 

\begin{theorem}\label{th:OPVB}
Let $\Omega$ be a relatively compact, strongly pseudoconvex domain with $\Cscr^2$ boundary in a Stein manifold.
The following hold for every real number $r>0$.
\begin{enumerate}[\rm (i)]
\item Every topological complex vector bundle on $\overline\Omega$
is isomorphic to a vector bundle of class 
$\lambda^r_\Oscr(\overline\Omega)$.
\item If two vector bundles of class $\lambda^r_\Oscr(\overline\Omega)$ 
are isomorphic as topological complex vector bundles, then they are also
isomorphic as $\lambda^r_\Oscr(\overline\Omega)$ vector bundles.
\end{enumerate}
\end{theorem}

See Theorem \ref{th:OPprincipalbundles} for the analogous result 
for principal fibre bundles. This result is classical for vector bundles on Stein spaces;
see Grauert \cite{Grauert1958MA}, with the special case 
of line bundles due to Oka \cite{Oka1939}. 
(See also Leiterer \cite{Leiterer1990} and 
\cite[Theorem 5.3.1]{Forstneric2017E}.) For vector bundles
of classes $\Ascr^k(\overline\Omega)$, $k\in \Z_+$,  
see Leiterer \cite{Leiterer1976I,Leiterer1976II}
and Heunemann \cite[Theorem 2]{Heunemann1986I}.

%
%
%
%
\section{Preliminaries on Zygmund spaces}\label{sec:prelim}

In the first two subsections, we recall the definition and basic properties of 
the big Zygmund spaces $\Lambda^r$ and the little Zygmund spaces 
$\lambda^r$ for any real number $r>0$. 
Manifold-valued maps are introduced in Subsect.~\ref{ss:manifold-valued}, and holomorphic mapping spaces and fibre bundles are discussed in Subsect.~\ref{ss:Lr-bundles}. In Subsect.~\ref{ss:dibar} 
we recall the relevant results on the solution to the $\dibar$-equation in 
the spaces $\Lambda^r$. For more information, see 
Gong \cite[Sect.\ 5]{Gong2025TAMS}, Wallin \cite{Wallin1988}, 
and the monographs by Gilbarg and Trudinger \cite[Sect. 4.1]{GilbargTrudinger1983},
E.\ M.\ Stein \cite[Sect.\ V.4]{SteinEM1970}, 
and M.\ E.\ Taylor \cite{TaylorME2023}.

%
%
\subsection{Zygmund spaces} \label{ss:HZspaces}

Let $\Omega$ be a domain in $\R^n$. 
Given a function $f:\Omega\to\C$ and a vector 
$h\in \R^n\setminus\{0\}$, the first and the second difference of $f$ 
at $x\in \Omega$, with step $h$, are defined by 
\begin{eqnarray}\label{eq:Delta}
	\Delta_h f(x) &=& f(x+h)-f(x), 
	\quad \ \ \quad\qquad\qquad x+h\in \Omega; \\
	\label{eq:Delta2}
	\Delta^2_h f(x) &=& f(x+h)+f(x-h)-2f(x),\quad x+h,\ x-h\in \Omega.
\end{eqnarray}
Note that $\Delta^2_h f(x) =\Delta_h f(x) - \Delta_h f(x-h)$.
(Some authors use the definition 
$\Delta^2_h f(x)=\Delta_h \circ\Delta_h f(x) = f(x+2h)+f(x)-2f(x+h)$, 
which leads to the same function spaces.)
Given $x\in\Omega$ and $0<\alpha\le 1$, the function $f$ 
belongs to $\Cscr^{0,\alpha}(\Omega,x)$ if 
\[ 
	[f]_{\alpha,x} := 
	\sup_{h\ne 0,\ x+h\in \Omega} |h|^{-\alpha}|\Delta_h f(x)|
	= \sup_{y\in\Omega\setminus\{x\}} \frac{|f(y)-f(x)|}{|y-x|^\alpha}  
	< \infty.
\] 
Such a function $f$ is said to be H\"older continuous of exponent 
$\alpha$ at the point $x$. For $\alpha=1$, 
$\Cscr^{0,\alpha}(\Omega,x)=\Lip^1(\Omega,x)$ is the 
Lipschitz space at $x$. The H\"older-$\alpha$ space on $\Omega$ is
\begin{equation}\label{eq:Calpha}
	\Cscr^{0,\alpha}(\Omega) =
	\bigl\{f:\Omega\to \C: \|f\|_{\Cscr^{0,\alpha}(\Omega)} = 
	\sup_{x\in \Omega} |f(x)| + \sup_{x\in \Omega}\, [f]_{\alpha,x}<\infty
	\bigr\}.
\end{equation}
For $\alpha=1$ we have the Lipschitz space
$\Cscr^{0,1}(\Omega)=\mathrm{Lip}^1(\Omega)$. 

Let $x=(x_1,\ldots,x_n)$ denote the coordinates on $\R^n$.
Given $\beta=(\beta_1,\ldots,\beta_n)\in\Z_+^n$, set 
$|\beta|=\beta_1+\cdots+\beta_n$ and  
$D^\beta f=\di^{|\beta|}f/\di x_1^{\beta_1} \cdots \di x_n^{\beta_n}$. 
For $k\in \Z_+$ we denote by $\Cscr^k(\Omega)$ the space
of $k$-times continuously differentiable functions $f$ on $\Omega$ with 
\[
	\|f\|_{\Cscr^k(\Omega)} 
	= \sum_{|\beta|\le k} \sup_{x\in \Omega} |D^\beta f(x)| < \infty.
\]
The H\"older space $\Cscr^{k,\alpha}(\Omega)$ for $k\in \Z_+$ and 
$0<\alpha\le 1$ is defined by 
\begin{eqnarray}\label{eq:Ckalpha}
	\Cscr^{k,\alpha}(\Omega) &=& 
	\Big\{f\in \Cscr^k(\Omega): D^\beta f\in \Cscr^{0,\alpha}(\Omega)
	\ \ \text{for all $\beta\in \Z_+^n$ with $|\beta|\le k$}, \\ 
	&& \quad 
	\label{eq:Ckalpha-norm}
	\|f\|_{\Cscr^{k,\alpha}(\Omega)} 
	:= 	\|f\|_{\Cscr^{k}(\Omega)} 
	+ \sum_{|\beta|=k} \|D^\beta f\|_{\Cscr^{0,\alpha}(\Omega)} <\infty 
	\Big\}.
\end{eqnarray}
For $\alpha=1$ one also writes 
$\Cscr^{k,1}(\Omega)= \mathrm{Lip}^{k,1}(\Omega)$.

Assume now that $\Omega\Subset \R^n$ is a bounded 
Lipschitz domain. This means that its boundary $b\Omega$
is locally at each point a Lipschitz graph over an 
affine hyperplane in $\R^n$, with the domain $\Omega$ lying on one side
of the graph. A function $f:\overline \Omega\to\C$ 
belongs to $\Cscr^k(\overline \Omega)$ for some $k\in\Z_+$ if it 
is the restriction to $\overline \Omega$ of a function 
$\tilde f\in \Cscr^k(\R^n)$. (For $k=0$, this coincides with the usual
definition of continuous functions by Tietze's extension theorem.) 
Given $\beta\in \Z_+^n$ with $|\beta|\le k$, we denote by 
$D^\beta f$ the restriction of $D^\beta \tilde f$ to $\overline \Omega$;
note that $D^\beta f$ is independent of the choice of the extension $\tilde f$
if $\Omega$ is a Lipschitz domain. 
Given $0<\alpha\le 1$, the H\"older space 
$\Cscr^{k,\alpha}(\overline \Omega)$ on the closed domain 
$\overline\Omega$ is defined by
\[
	\Cscr^{k,\alpha}(\overline \Omega)=
	\big\{f\in \Cscr^k(\overline \Omega):
	D^\beta f \in \Cscr^{0,\alpha}(\overline \Omega)\ \ 
	\text{for all $\beta\in \Z_+^n$ with $|\beta|\le k$} \big\},
\]
endowed with the norm $\|f\|_{\Cscr^{k,\alpha}(\Omega)}$ 
\eqref{eq:Ckalpha-norm}. Since $\Omega$ is Lipschitz, 
we get the same norm by taking the suprema over $x\in \overline\Omega$, so 
$\|f\|_{\Cscr^{k,\alpha}(\Omega)}=\|f\|_{\Cscr^{k,\alpha}(\overline\Omega)}$.

%
%
\begin{proposition}
\label{prop:extension}
If $\Omega$ is a bounded Lipschitz domain in $\R^n$
then every function $f\in \Cscr^{k,\alpha}(\Omega)$
$(k\in \Z_+,\ 0<\alpha\le 1)$ extends to a function 
$F\in \Cscr^{k,\alpha}(\R^n)$, and $F|_{\overline\Omega}$ is unique.
\end{proposition}

\begin{proof}
For $k=0$ this holds by McShane's extension theorem  
\cite[Corollary 1, p.\ 840]{McShane1934}. 

Suppose that $k\ge1$. For every multi-index $\beta$ with $|\beta|\le k$, 
the function $D^\beta f$ is uniformly H\"older continuous on $\Omega$, and 
hence extends uniquely to a function $g_\beta$ on $\overline\Omega$ 
with the same H\"older estimate.
We claim that the functions $g_\beta$ 
on $\overline\Omega$ form a Whitney jet of order $k$.
To see this, recall that sufficiently close points $x,y\in\Omega$
in a bounded Lipschitz domain can be joined in $\Omega$ by a rectifiable path 
of length at most $C|x-y|$, contained in a ball of radius $C|x-y|$ about $x$. 
Successive integration of derivatives along such paths gives 
\[ 
	\bigg| D^\beta f(y) - \sum_{|\gamma|\le k-|\beta|} 
	\frac{D^{\beta+\gamma}f(x)}{\gamma!}(y-x)^\gamma \bigg| 
	\le C_f |x-y|^{k-|\beta|+\alpha}, \qquad |\beta|\le k. 
\] 
Passing to limits gives the same estimates for the functions $g_\beta$ on 
$\overline\Omega$, thus justifying the claim. The H\"older version of the 
Whitney extension theorem \cite{Whitney1934TAMS}
(see also \cite[Theorem 3.2 and Comp.\ 3.5]{Malgrange1966})
gives an extension $F$ of class $\Cscr^{k,\alpha}$ on a neighbourhood 
of $\overline\Omega$ satisfying $D^\beta F = g_\beta$ on $\overline \Omega$ 
for all $|\beta|\le k$. Multiplication by a smooth cutoff equal to $1$ near 
$\overline\Omega$ gives $F\in\Cscr^{k,\alpha}(\R^n)$. 
The restriction $F|_{\overline\Omega}$ 
is unique because $\Omega$ is dense in $\overline\Omega$.
\end{proof}

We now recall the definition of Zygmund spaces $\Lambda^r(\Omega)$ 
for any real number $r>0$. Write $r=k+\alpha$ with $k\in\Z_+$ and 
$0<\alpha\le 1$. If $r$ is not an integer (that is, $\alpha\ne 1$), set 
$\Lambda^r(\Omega)=\Cscr^{k,\alpha}(\Omega)$ and 
$\Lambda^r(\overline \Omega)=\Cscr^{k,\alpha}(\overline \Omega)$;
both spaces are endowed with the norm 
$\|f\|_{\Lambda^{k+\alpha}(\Omega)}:=\|f\|_{\Cscr^{k,\alpha}(\Omega)}$
\eqref{eq:Ckalpha-norm}. Assume now that $r\in \N=\{1,2,\ldots\}$.  
Recall that the second difference $\Delta_h^2 f$ is given 
by \eqref{eq:Delta2}. We define
\begin{eqnarray}\label{eq:Lambda-1}
	\Lambda^1(\Omega) &=& \big\{f\in \Cscr(\Omega):
	\|f\|_{\Lambda^1(\Omega)} = 
	\sup_{x\in\Omega}|f(x)| + \sup_{x,x\pm h\in\Omega,\, h\ne 0}
	|h|^{-1} |\Delta_h^2 f(x)| <\infty \big\}, \\
	\label{eq:Lambda-r}
	\Lambda^r(\Omega) &=&
	\big\{f\in \Cscr^{r-1}(\Omega):
	\|f\|_{\Lambda^r(\Omega)} = \|f\|_{\Cscr^{r-1}(\Omega)} +
	\sum_{|\beta|=r-1} \|D^\beta f\|_{\Lambda^1(\Omega)} <\infty\big\},
\end{eqnarray}
where $r>1$ in \eqref{eq:Lambda-r}. 
(For $\Omega=\R^n$, see 
\cite[Propositions 8 and 9, pp.\ 146--147]{SteinEM1970}.)
The space $\Lambda^1(\R^n)$ is the classical 
Zygmund space on $\R^n$; omitting the term 
$\|f\|_\infty=\sup_{x\in\Omega}|f(x)|$ in \eqref{eq:Lambda-1}
gives the homogeneous Zygmund space.
Clearly, $\Lambda^1(\R^n)$ contains the Lipschitz space 
$\Cscr^{0,1}(\R^n) = \mathrm{Lip}^1(\R^n)$
but is not equal to it as shown by
\cite[Example 4.3.1, p.\ 148]{SteinEM1970}.
The norms \eqref{eq:Lambda-1}--\eqref{eq:Lambda-r} are called the
{\em interior Zygmund norms}; the suprema use admissible triples only.
An equivalent and convenient definition, valid for every $r>0$, uses restrictions of functions on $\R^n$; see \cite[Definition 8, p.~105]{Wallin1988}:
\begin{eqnarray}\label{eq:Lambda-r-Omega}
	\Lambda^r(\overline \Omega) &=& 
	\big\{f=\tilde f|_{\overline \Omega}: \tilde f\in  \Lambda^r(\R^n)\big\}, \\
	\|f\|_{\Lambda^r(\overline\Omega)} &=& 
	\label{eq:outer}
	\inf\big\{ \|\tilde f\|_{\Lambda^r(\R^n)}: \tilde f \in \Lambda^r(\R^n),\
	 \tilde f|_{\overline \Omega} =f\big\}. 
\end{eqnarray}
The norms \eqref{eq:outer} are called {\em exterior Zygmund norms} 
on $\Omega$. By a classical theorem due to Besov, Ilin and Nikolski 
\cite[Theorem 18.5, p.\ 63]{BesovIlinNikolskii1979},
every function in $\Lambda^r(\Omega)$ extends to a function in 
$\Lambda^r(\R^n)$.
Furthermore, the interior and exterior definitions of the spaces 
$\Lambda^r(\overline\Omega)$ and their norms are equivalent. 
Indeed, E.\ M.\ Stein \cite[Sect.\ VI.2]{SteinEM1970}
constructed a linear extension operator 
$E:\Cscr(\overline\Omega) \to \Cscr_0(\R^n)$ 
to the space of continuous functions on $\R^n$ with compact support 
such that $E:\Lambda^r(\overline\Omega)\to \Lambda^r(\R^n)$ 
is bounded for every $r>0$. Stein proved that his extension operator 
is bounded on Sobolev spaces; for Zygmund spaces 
this was shown by Gong \cite[Corollary 5.3]{Gong2025TAMS}.
Furthermore, Shi and Yao \cite[Theorem 1.1]{ShiYao2024} 
proved that for $r=k+\alpha$ with $k\in \N$ and $0<\alpha\le 1$,
$\|f\|_{\Lambda^r(\Omega)}$ is equivalent to
$\sum_{j\le k} \|\nabla^j f\|_{\Lambda^\alpha(\Omega)}$.
Their proof uses Rychkov's universal extension 
for Besov spaces $B^r_{p,q}(\Omega)$; see \cite{Rychkov1999}
and note that $\Lambda^r(\Omega)=B^r_{\infty,\infty}(\Omega)$. 
For noninteger values $r>0$ and domains $\Omega$ 
with smooth boundaries, see also 
\cite[Lemma 6.37]{GilbargTrudinger1983} whose 
proof is based on Seeley's extension theorem \cite{Seeley1964}. 	
If $X$ is a smooth manifold and $\Omega$ is a Lipschitz domain in $X$ 
with compact closure, one defines the spaces $\Lambda^{r}(\Omega)$ 
and $\Lambda^{r}(\overline \Omega)$, $r>0$,
by using a finite system of smooth coordinate 
charts on $X$ covering $\overline \Omega$. 
The norms obtained in this way are comparable to one another;
see (3.11) and (3.12) in \cite[Lemma 3.1]{GanGong2024}.
(See also \cite{Palais1968,Palais1971} 
and the discussion and references in 
\cite{BourdaudCristoforis2002}, \cite[Sect.\ 2]{Forstneric2007AJM}, 
and \cite[Subsect.\ 3.2]{GongShi2024}.)
The space $\Lambda^r(\overline\Omega)$, $r>0$, 
is a commutative unital Banach algebra 
under pointwise multiplication, with Moser-type estimates for products, 
compositions, and inverses; see \cite[Lemmas 3.1--3.3]{Gong2019MA},
\cite[Lemma 6.3]{Gong2025TAMS}, and \cite{BourdaudCristoforis2002}. 
There are continuous strict embeddings among the 
classical and the Zygmund scales on
any relatively compact Lipschitz domain $\Omega\Subset X$
in a smooth manifold $X$:
\begin{equation}\label{eq:scale}
	\Cscr^{k+1}\subset \Cscr^{k,1} \subset \Lambda^{k+1}\subset
	\Lambda^{k+\alpha} \subset 
	\Lambda^{k+\beta}\subset \Cscr^k,\quad 
	k\in\Z_+,\ 0<\beta<\alpha<1.
\end{equation}

\begin{definition}\label{def:local}
A function $f:X\to \C$ on a smooth manifold $X$ 
is said to be of {\em local class $\Lambda^r$}, $r>0$,  
if every point $x\in X$ has a compact Lipschitz neighbourhood
$K \subset X$ such that $f|_K \in \Lambda^r(K)$.
\end{definition}

The space $\Lambda^{r,loc}(X)$ of all functions of local class $\Lambda^r$ 
is a Fr\'echet space, and a Banach space if $X$ is compact 
(possibly with boundary). In the latter case, 
$\Lambda^{r,loc}(X)=\Lambda^{r}(X)$. 

%
%
\subsection{Little Zygmund spaces}
\label{ss:little}
For a Lipschitz domain $\Omega\Subset X$ and $r>0$, define
\[
 	\lambda^r(\overline \Omega)=
	\overline{\Cscr^\infty(\overline\Omega)}^{\,\Lambda^r(\overline\Omega)},
\]
the closure of the space of smooth functions on $\overline \Omega$
in $\Lambda^r(\overline\Omega)$. 
(Here $\Cscr^\infty(\overline\Omega)$ denotes restrictions of functions 
smooth on a neighbourhood of $\overline\Omega$.)
This is a closed Banach subalgebra of 
$\Lambda^r(\overline\Omega)$ with the inherited norm  
$\|f\|_{\lambda^r}=\|f\|_{\Lambda^r}$.
For noninteger $r=k+\alpha$, $0<\alpha<1$, this is the little
H\"older space $c^{k,\alpha}(\overline\Omega)$. 
The condition can be imposed locally in smooth coordinates 
in $X$, and is invariant under smooth changes of coordinates and smooth
superposition. For manifold-valued maps, bundle transition maps,
and sections of bundles, the little condition is imposed on local 
coordinate expressions. On a domain $\Omega\Subset \R^n$
and for $f\in\Cscr^{k,\alpha}(\overline\Omega)$, we have
\[ 
	f\in\lambda^{k+\alpha}(\overline\Omega) \quad\Longleftrightarrow\quad 	
	\lim_{\delta\downarrow0} 
	\sup_{\substack{x,x+h\in\overline \Omega\\0<|h|\le\delta}} 	
	\frac{|D^\beta f(x+h)-D^\beta f(x)|}{|h|^\alpha}=0 
	\quad\text{for every }|\beta|=k. 
\] 
At integer exponents, use second differences of derivatives of order $r-1$,
or, equivalently, differences of any integer order $m>r$ divided
by $|h|^r$ in local extension charts.

%
%
\begin{lemma}\label{lem:interpolation-dilation}
For a compact Lipschitz domain $\overline \Omega$ and 
numbers $0<s<r$ there is a constant $C>0$ such that for every 
$f\in\Lambda^r(\overline\Omega)$,
\begin{equation}\label{eq:interpolation}
 	\|f\|_{\Lambda^s(\overline \Omega)}\le
 	C\|f\|_{\Cscr^0(\overline \Omega)}^{1-s/r} 
	\|f\|_{\Lambda^r(\overline \Omega)}^{s/r}.
\end{equation}
In particular there is a continuous inclusion
$\Lambda^r(\overline \Omega)\hookrightarrow\lambda^s(\overline \Omega)$.
If $\overline \Omega\subset\R^n$ is a compact convex domain 
containing $0$ in its interior then the dilation operators 
$R_t(f)(x)=f(tx)$ for $0<t<1$ are uniformly bounded on $\Lambda^r(\overline \Omega)$, 
and $\lim_{t\to1} R_t(f) = f$ in $\Lambda^s(\overline \Omega)$ 
for $0<s<r$. For $f\in\lambda^r(\overline \Omega)$ convergence holds in 
$\Lambda^r(\overline \Omega)$.
Convergence is uniform on compact families in the respective
source spaces. The assertions persist under fixed smooth
coordinate changes.
\end{lemma}

Lemma \ref{lem:interpolation-dilation} follows from the finite-difference 
characterisation of Zygmund spaces and the extension property recalled in 
Subsect.~\ref{ss:HZspaces}. For completeness, we give the details. 

\begin{proof} For a bounded function $f$ on $\R^n$ and $m\in\N$, write 
\[ 
	\delta_h^m f(x) =\sum_{j=0}^m(-1)^{m-j}\binom mj f(x+jh), 
	\qquad h\in\R^n. 
\] 
Recall that, for every $a>0$ and every integer $m>a$, an equivalent norm on 
$\Lambda^a(\R^n)$ is 
\begin{equation}\label{eq:Na}
	N_a(f) =\|f\|_\infty+ \sup_{h\ne0}|h|^{-a}\|\delta_h^m f\|_\infty. 
\end{equation}

\smallskip \noindent\emph{Proof of interpolation \eqref{eq:interpolation}.} 
Fix an integer $m>r$, and use this same difference order for the exponents $s$ 
and $r$. Set $A=\|f\|_\infty$ and $B=N_r(f)$. For every $h\ne0$ we have 
\[ 
	\|\delta_h^m f\|_\infty \le \min\{2^m A,B|h|^r\}. 
\] 
Using 
\[ 
	\min\{a,b\}\le a^{1-\theta}b^\theta, \qquad 
	a,b\ge0,\quad 0<\theta<1, 
\] 
with $\theta=s/r$, we obtain 
\[ 
	|h|^{-s}\|\delta_h^m f\|_\infty \le 2^{m(1-s/r)}A^{1-s/r}B^{s/r}. 
\] 
Since $A\le B$, the same bound, up to a constant, controls the zeroth-order term. Consequently, 
\[ 
	\|f\|_{\Lambda^s(\R^n)} \le 
	C\|f\|_\infty^{1-s/r} \|f\|_{\Lambda^r(\R^n)}^{s/r}. 
\] 
Suppose first that $\overline\Omega\subset\R^n$. Choose a linear extension 
operator $E$ which is bounded simultaneously in the uniform norm and in the 
$\Lambda^r$ norm: 
\[ 
	\|Ef\|_\infty\le C_0\|f\|_{\Cscr^0(\overline\Omega)}, 
	\qquad \|Ef\|_{\Lambda^r(\R^n)} \le C_r\|f\|_{\Lambda^r(\overline\Omega)}. 
\] 
Such an operator exists for a bounded Lipschitz domain. Applying the preceding estimate to $Ef$ and restricting to $\overline\Omega$ proves \eqref{eq:interpolation}. For a compact Lipschitz domain in a smooth manifold, the same argument applies in 
finitely many charts, using a fixed smooth partition of unity. 

\smallskip \noindent\emph{The inclusion 
$\Lambda^r(\overline \Omega)\hookrightarrow\lambda^s(\overline \Omega)$.} 
In the Euclidean case $\overline \Omega\subset\R^n$, let $F=Ef$, and let 
\[ 
	F_\epsilon=\rho_\epsilon*F 
\] 
be its regularisations by a nonnegative smooth mollifier of integral one. Then 
\[ 
	\sup_{0<\epsilon\le1} \|F_\epsilon\|_{\Lambda^r(\R^n)} 
	\le C\|F\|_{\Lambda^r(\R^n)}, 
	\qquad \lim_{\epsilon\to 0}\|F_\epsilon-F\|_\infty = 0. 
\] 
Here uniform convergence follows from the uniform continuity of $F$; 
an extension with compact support may be used. Applying interpolation to 
$F_\epsilon-F$ gives 
\[ 
	\lim_{\epsilon\to 0} \|F_\epsilon-F\|_{\Lambda^s(\R^n)} = 0. 
\] 
Thus $f$ is a $\Lambda^s(\overline\Omega)$ limit of smooth functions, 
so $f\in\lambda^s(\overline\Omega)$. The same conclusion on a manifold 
follows by local regularisation and a partition of unity. Moreover,
\eqref{eq:interpolation} and the continuous embedding 
$\Lambda^r\hookrightarrow\Cscr^0$ give 
\[ 
	\|f\|_{\lambda^s(\overline\Omega)} =
	\|f\|_{\Lambda^s(\overline\Omega)} 
	\le C\|f\|_{\Lambda^r(\overline\Omega)}. 
\] 
Hence the inclusion $\Lambda^r(\overline\Omega)\hookrightarrow \lambda^s(\overline\Omega)$ is continuous. 

\smallskip \noindent\emph{Dilation.} 
Let $K=\overline \Omega \subset\R^n$ be as in the statement, and put 
\[
	R_t f(x)=f(tx)\ \ \text{for $f\in\Lambda^r(K)$, $x\in K$, and $0<t\le1$}. 
\]
If $F$ is any $\Lambda^r(\R^n)$ 
extension of $f$, then $F(t\,\cdot)$ extends $R_t f$, and 
\[ 
	\delta_h^m\bigl(F(t\,\cdot)\bigr)(x) =(\delta_{th}^mF)(tx). 
\] 
Recall that $N_r$ is defined by \eqref{eq:Na}. It follows that 
\[ 
	N_r\bigl(F(t\,\cdot)\bigr) \le 
	\|F\|_\infty+ t^r\sup_{h\ne0}\left(|h|^{-r}\|\delta_h^mF\|_\infty\right) 
	\le N_r(F). 
\] 
Taking the infimum over extensions and using equivalence of norms yields 
\[ 
	\sup_{0<t\le1}\|R_t f\|_{\Lambda^r(K)} \le M_r\|f\|_{\Lambda^r(K)},
\] 
so the dilation operators are uniformly bounded.  
Since $f$ is uniformly continuous on $K$, 
\[ 
	\lim_{t\uparrow1} \|R_t f-f\|_{\Cscr^0(K)} = 0. 
\] 
The uniform $\Lambda^r$ bound and \eqref{eq:interpolation} therefore imply 
\[ 
	\lim_{t\uparrow1} \|R_t f-f\|_{\Lambda^s(K)}=0, \qquad 0<s<r.
\] 
Now suppose that $f\in\lambda^r(K)$. Given $\epsilon>0$, choose 
$g\in\Cscr^\infty(K)$ such that $\|f-g\|_{\Lambda^r(K)}<\epsilon$. 
For smooth $g$, the chain rule and uniform continuity of its derivatives show 
that $R_tg\to g$ in $\Cscr^q(K)$ as $t\uparrow1$ for every integer $q\ge0$. 
Taking $q>r$ gives convergence in $\Lambda^r(K)$. Consequently, 
\[ 
	\begin{split} \|R_t f-f\|_{\Lambda^r(K)} &\le 
	\|R_t(f-g)\|_{\Lambda^r(K)} +\|R_tg-g\|_{\Lambda^r(K)} 
	+\|g-f\|_{\Lambda^r(K)}\\ &\le (M_r+1)\epsilon
	+ \|R_tg-g\|_{\Lambda^r(K)}. \end{split} 
\] 
Letting $t\uparrow1$ and then $\epsilon\downarrow0$ proves convergence 
at exponent $r$. 

\smallskip \noindent\emph{Compact families and coordinate changes.} 
For lower-exponent convergence, regard $R_t-I$ as an operator  
$\Lambda^r(K) \to \Lambda^s(K)$. For same-exponent convergence, 
regard it as an operator $\lambda^r(K)\to \Lambda^r(K)$. 
In either case these operators have uniformly bounded norms for $t$ near $1$ 
and converge to zero as $t\uparrow 1$
on every element of the source space. Let $\mathcal F$ 
be a compact subset of the relevant source space. Given $\epsilon>0$, 
choose a finite $\epsilon$-net $f_1,\ldots,f_N$ in $\mathcal F$. If 
$L$ is a common bound for the operator norms, then 
\[ 
	\sup_{f\in\mathcal F}\|(R_t-I)f\|_{\mathrm{target}} 
	\le L\epsilon+ \max_{1\le j\le N} \|(R_t-I)f_j\|_{\mathrm{target}}. 
\] 
This proves uniform convergence on $\mathcal F$. Finally, pullback by a  
smooth coordinate diffeomorphism, defined on a neighbourhood of $K$,   
is a bounded isomorphism on each $\Lambda^a$ space. It also 
preserves $\lambda^a$, since it maps smooth functions to smooth functions.
Conjugating the dilation operators by these pullbacks proves the assertions 
in fixed smooth coordinates. 
\end{proof}

The following examples illustrate the difference between the little and big  
Zygmund spaces. 

\begin{example}\label{ex:ex1}
The function \[ f(x)=\left|x-\frac12\right|, \qquad x\in[0,1], \] belongs to 
$\Lambda^1([0,1])$, but cannot be approximated by functions in 
$\Cscr^\infty([0,1])$ in the $\Lambda^1$ norm. 
Indeed, since $f$ is Lipschitz with Lipschitz constant 
$1$, \[ |f(x+h)-2f(x)+f(x-h)|\le 2h \] whenever $h>0$ and $x-h,x+h\in[0,1]$. 
Hence $[f]_{\Lambda^1([0,1])}\le 2$. At $x=\frac12$, we have 
\[ 
	f\left(\frac12+h\right)-2f\left(\frac12\right) +f\left(\frac12-h\right)=2h, 
	\qquad 0<h\le\frac12, 
\] 
so in fact $[f]_{\Lambda^1([0,1])}=2$. On the other hand, for  
$g\in \Cscr^1([0,1])$ Taylor's formula gives 
\[ 
	\frac{ g(\frac12+h)-2g(\frac12)+g(\frac12-h) }{h}\longrightarrow 0 
	\qquad\text{as }h\downarrow0. 
\] 
Consequently, $[f-g]_{\Lambda^1([0,1])} \ge 2$ for every $g\in \Cscr^1([0,1])$. 
\end{example}

\begin{example}\label{ex:density-obstruction}
On the unit disc $\D\subset\C$, for $0<\alpha<1$, the 
holomorphic function $f(z)=(1-z)^\alpha$ is in $\Lambda^\alpha(\overline \D)$ 
but not in $\lambda^\alpha(\overline \D)$. 
(We use the branch of the power function on the right half-plane which is 
positive on the positive real axis, and extend $f$ continuously by $f(1)=0$.)
Indeed, 
\[
 	\frac{|f(1-h)-f(1)|}{h^\alpha}=1, \qquad 0<h<1,
\]
but for smooth functions the quotient tends to zero as $h\to 0$. 
Thus, 
we cannot approximate $f$ in $\Lambda^\alpha(\overline \D)$
by holomorphic functions on neighbourhoods of $\overline\D$.
\end{example}

%
%
\subsection{Manifold-valued maps of Zygmund classes}
\label{ss:manifold-valued} 

Let $\Omega\Subset X$ be a Lipschitz domain in a smooth
manifold $X$, and let $Y$ be another smooth manifold.
We wish to define the space 
$
	\Lambda^r(\overline \Omega,Y)
$
of maps $f:\overline\Omega\to Y$ of Zygmund class 
$\Lambda^r(\overline\Omega, Y)$.
For this, we need to know that postcompositions of Zygmund 
class maps between Euclidean spaces by smooth maps 
preserve these classes. More precisely, if 
$\Omega\Subset\R^n$ is a Lipschitz domain and 
$u=(u_1,\ldots,u_m):\overline \Omega\to \R^m$, 
with $u_i\in \Lambda^r(\overline\Omega)$ for $i=1,\ldots,m$, 
and $F$ is a smooth function on a neighbourhood of  
$u(\overline \Omega)$ in $\R^m$, we must know that 
$F\circ u \in \Lambda^r(\overline \Omega)$.
For noninteger $r$ (i.e., for H\"older spaces), this is 
classical \cite{GilbargTrudinger1983}.
When $r$ is an integer and $n=1$, 
a suitable reference is \cite{BourdaudCristoforis2002}. 
A precise reference for the general vector-valued case 
with integer $r$ seems difficult to find.
For this reason, we include the following lemma.

%
%
\begin{lemma}\label{lem:composition}
Let $\Omega\Subset\R^n$ and $D\Subset \R^m$ be Lipschitz domains, 
$u=(u_1,\ldots,u_m): \overline\Omega\to \bar D$ a map of class 
$\Lambda^r(\overline \Omega)$ for some $r\in \N$, and 
$F: \bar D \to \mathbb R$ a function of class $\Cscr^{r+1}(\overline D)$. 
Then, $F\circ u\in \Lambda^r(\overline \Omega)$.
\end{lemma}

\begin{proof}
Using the Stein extension operator 
\cite[Sect.\ VI.2]{SteinEM1970} and the results 
mentioned above, we may assume $F\in \Cscr^{r+1}(\R^m)$
and $u\in \Lambda^r(\R^n)$ both have compact support.
By $\|u\|_{\Lambda^r}$ we denote the Zygmund norm 
of the extended function, and likewise for $F$.

We first prove the lemma for $r=1$. Given points $x,y\in\R^n$,
consider the function 
\[
 	g(t)=F\big((1-t)u(x)+tu(x+y)\big) +F\big((1-t)u(x)+tu(x-y)\big).
\]
We have 
\[
	g(1)-g(0) = F\big(u(x+y)\big)+ F\big(u(x-y)\big)-2F(u(x)) 
	= \Delta^2_y (F\circ u)(x).
\]
(See \eqref{eq:Delta2}.) Use $g(1)-g(0)=\int_0^1g'(t)\,dt$ and estimate 
$g'(t)$ uniformly in $t$. We have
\begin{eqnarray*}
   g'(t) &=& \nabla F\big((1-t)u(x)+tu(x+y)\big) \big(u(x+y)-u(x)\big)\\
   && +\nabla F\big((1-t)u(x)+tu(x-y)\big) \big(u(x-y)-u(x)\big) \\
   &=& \nabla F\big((1-t)u(x)+tu(x+y)\big)\Delta^2_y u(x) \\
   && + \big[\nabla F\big((1-t)u(x)+tu(x-y)\big) 
   	- \nabla F\big((1-t)u(x)+tu(x+y)\big)\big] 
   				\big(u(x-y)-u(x)\big). 
\end{eqnarray*}
The first term after the last equality sign is bounded by 
$\|F\|_{\Cscr^1}\|u\|_{\Lambda^1}|y|$. In the second term,
\begin{eqnarray*}
	&& \big|\nabla F((1-t)u(x)+tu(x-y)) - \nabla F((1-t)u(x)+tu(x+y))\big| 
	\\
	&& \quad \le const \|F\|_{\Cscr^2} |u(x+y)-u(x-y)| 
	\le const \|F\|_{\Cscr^2} \|u\|_{\Lambda^\alpha}|y|^\alpha
\end{eqnarray*}
for any $\alpha \in (0,1)$, where the constant only depends on 
$\alpha$ and $m$. (We used that $\Lambda^1\subset \Lambda^\alpha$.)
Similarly, we estimate the last factor in the second term by 
\[
	|u(x-y)-u(x)| \le \|u\|_{\Lambda^{1-\alpha}} |y|^{1-\alpha}.
\]
Hence, the second term in the expression for $g'(t)$ is bounded by 
\[
	const \|F\|_{\Cscr^2} \|u\|_{\Lambda^\alpha}|y|^\alpha
	 \|u\|_{\Lambda^{1-\alpha}}|y|^{1-\alpha}
	 = const \|F\|_{\Cscr^2} \|u\|_{\Lambda^\alpha}
	 \|u\|_{\Lambda^{1-\alpha}} |y|.
\]
For $y \ne 0$ it follows that 
\[ 
	|y|^{-1} |\Delta^2_y (F\circ u)(x)| \le  
	\|F\|_{\Cscr^1}\|u\|_{\Lambda^1} +
	const \|F\|_{\Cscr^2} \|u\|_{\Lambda^\alpha}
	\|u\|_{\Lambda^{1-\alpha}},
\]
and hence $F\circ u\in \Lambda^1$. This proves the result for $r=1$. 
Assume now that $r>1$ and the lemma holds for functions in $\Lambda^{r-1}$.
Let $F\in \Cscr^{r+1}$ and $u=(u_1,\ldots,u_m)\in \Lambda^{r}$.
Denoting the coordinates on $\R^n$ and $\R^m$ 
by $x=(x_1,\ldots,x_n)$ and $y=(y_1,\ldots,y_m)$, respectively, we have
\[
	\frac{\di}{\di x_i} F(u(x)) = 
	\sum_{j=1}^m \frac{\di F}{\di y_j}(u(x)) \frac{\di u_j}{\di x_i}(x),
	\quad i=1,\ldots,n.
\]
Since $\frac{\di F}{\di y_j}\in \Cscr^{r}$ for 
$j=1,\ldots,m$, the inductive assumption implies 
$\frac{\di F}{\di y_j}\circ u \in \Lambda^{r-1}$.
Furthermore, $\frac{\di u_j}{\di x_i}\in \Lambda^{r-1}$,
so $\frac{\di}{\di x_i} (F\circ u) \in \Lambda^{r-1}$.
It follows that $F\circ u\in \Lambda^{r}$.
This completes the induction.
\end{proof}

%
%
\begin{remark}\label{rem:composition}
With suitable adjustments, the proof of Lemma \ref{lem:composition}
also applies to classes $\Lambda^r$ with noninteger $r$,
that is, to H\"older classes. By elaborating on the proof, one can 
obtain explicit estimates and see in particular that the postcomposition 
by a smooth function on $\R^m$ defines a smooth operator 
$\Lambda^r(\overline \Omega,\R^m)\to \Lambda^r(\overline \Omega)$. 
\end{remark}

Smooth superposition also preserves little Zygmund spaces: if $u_j\to u$ in 
$\Lambda^r$ with $u_j$ smooth, continuity of superposition gives 
$F\circ u_j\to F\circ u$ in $\Lambda^r$, and hence $F\circ u\in\lambda^r$.
Hence, Lemma \ref{lem:composition} lets us define the 
Zygmund mapping spaces 
\[
	\lambda^r(\overline \Omega,Y) \subset \Lambda^r(\overline \Omega,Y),
	\quad r>0,
\]
for any Lipschitz domain $\Omega\Subset X$ in a smooth
manifold $X$, using a system of smooth local charts on both 
manifolds. Likewise, we can define the spaces 
\[
	\lambda^{r,loc}(X,Y)\subset \Lambda^{r,loc}(X,Y)
\] 
of maps of local Zygmund classes between any pair of smooth manifolds.

%
%
\subsection{Fibre bundles of Zygmund classes} \label{ss:Lr-bundles}

On a smooth manifold $X$, one defines fibre bundles of classes
$\lambda^r$ and $\Lambda^r$ by asking that the transition maps are
of local class $\lambda^r$ and $\Lambda^r$, respectively.
Assume now that $X$ and $Y$ are complex manifolds and 
$\Omega\Subset X$ is an open Lipschitz domain. Set
\[
	\lambda^r_\Oscr(\overline\Omega,Y)=
	\{f\in \lambda^r(\overline\Omega,Y): 
	f\ \text{is holomorphic on}\ \Omega\},\quad r>0.
\]
In particular, 
$\lambda^r_\Oscr(\overline\Omega,\C) =\lambda^r_\Oscr(\overline\Omega)$.
A holomorphic fibre bundle $h:Z\to \overline\Omega$ 
of class $\lambda^r_\Oscr(\overline \Omega)$ with fibre $Y$  
has total space of the form 
\[
	Z= \bigsqcup_{i=1}^m U_i \times Y/\!\!\sim, 
\]  
where $\{U_i\}_{i=1}^m$ is an open covering of 
$\overline \Omega$ by relatively open Lipschitz domains  
$U_i\subset \overline \Omega$ (which can be chosen of
the form $U_i=\wt U_i\cap \overline \Omega$ with 
$\wt U_i$ an open Lipschitz domain in $\R^n$)
such that the intersections $U_{i,j}=U_i\cap U_j$ are also Lipschitz, 
and a point $(x,y)\in U_j \times Y$ $(x\in U_{i,j})$ 
is identified by the equivalence relation $\sim$ with the point
$(x,\phi_{i,j}(x,y))\in U_i\times Y$, where the transition map
\[
	U_{i,j} \times Y \ni (x,y) \longmapsto \phi_{i,j}(x,y)\in Y
\] 
is of class $\lambda^{r,loc}_\Oscr(U_{i,j})$ in $x$,
holomorphic in $y$, and $\phi_{i,j}(x,\cdotp)\in \Aut(Y)$ 
for every $x\in U_{i,j}$. (The regularity in $x$ is understood 
locally uniformly on compact fibre-coordinate subsets, with holomorphic 
dependence on the fibre variable. Equivalently, in local coordinates 
one may require joint local $\lambda^r$ regularity, together with the stated holomorphicity.)
A section $f:\overline \Omega\to Z$ of $h$ 
is given by a collection of maps $f_i:U_i\to Y$ satisfying the 
compatibility conditions
\[
	f_i(x) = \phi_{i,j}(x, f_j(x)),\quad x\in U_{i,j},\ i,j=1,\ldots,m.
\]
The section $f:\overline \Omega\to Z$ is of class 
$\lambda^r_\Oscr(\overline \Omega)$ if and
only if $f_i\in \lambda^{r,loc}_\Oscr(U_i,Y)$ for $i=1,\ldots,m$. 
We denote by $\Gamma^r_\Oscr(\overline \Omega,Z)$ the space 
of sections of class $\lambda^r_\Oscr(\overline \Omega)$
of the bundle $h:Z\to\overline\Omega$.
If $Y$ is a smooth manifold, the analogous definition gives
fibre bundles $Z\to \overline\Omega$ of class $\lambda^r(\overline\Omega)$.

In particular, denoting by $GL_n(\C)$ the 
group of invertible $n\times n$ complex matrices, 
a vector bundle $\pi:E\to\overline \Omega$
of class $\lambda^r_\Oscr(\overline \Omega)$ and rank $n$
has fibre $\C^n$ and transition maps 
\[
	\phi_{i,j}(x,v)=A_{i,j}(x)v, \quad v\in\C^n,
\]
where the maps $A_{i,j}\in \lambda^{r,loc}_\Oscr(U_{i,j},GL_n(\C))$
satisfy the 1-cocycle conditions
\[
	A_{i,i}=I,\quad A_{i,j}A_{j,i}=I,\quad 
	A_{i,j}A_{j,k}A_{k,i}=I\quad \text{for all}\ i,j,k=1,\ldots,m.
\]
Here, $I\in GL_n(\C)$ denotes the identity matrix.
By shrinking the sets $U_i$ in the covering, we can represent the same
vector bundle by transition maps 
$A_{i,j}\in \lambda^{r}_\Oscr(U_{i,j},GL_n(\C))$.
A section of $E$ of class $\lambda^r_\Oscr(\overline\Omega)$ 
is given by a collection of maps $f_i\in \lambda^{r,loc}_\Oscr(U_i,\C^n)$ 
satisfying the compatibility conditions $f_i=A_{i,j}f_j$ on $U_{i,j}$.
Similarly one defines vector bundle morphisms of class 
$\lambda^r_\Oscr(\overline \Omega)$, subbundles, quotient bundles, etc. 
We refer to Leiterer \cite{Leiterer1990} for a similar treatment of
vector bundles of class $\Ascr^r(\overline \Omega)$ 
for $r\in \Z_+$, i.e., of class $\Cscr^r(\overline \Omega)$ and
holomorphic on $\Omega$. The analogous constructions
apply to classes $\Lambda^r$ in place of $\lambda^r$.

%
%
\subsection{A linear solution operator for the $\dibar$-equation} 
\label{ss:dibar}
Write $\Lambda^r_{p,q}(\overline\Omega)$ for forms with coefficients
of class $\Lambda^r$ in smooth ambient frames, endowed with the 
$\Lambda^r$ topology.

\begin{theorem}\label{th:dibar}
Let $\Omega\Subset X$ be strongly pseudoconvex with $\Cscr^2$ boundary
in a Stein manifold of dimension $n$. For $0\le p\le n$ and
$1\le q\le n$ there is a linear right inverse $T_q$ to $\dibar$ with
bounded restrictions
\[
 	T_q:\Lambda^r_{p,q}(\overline\Omega)\cap\ker\dibar
 	\longrightarrow\Lambda^{r+1/2}_{p,q-1}(\overline\Omega),\qquad r>0.
\]
\end{theorem}

For $0<r\le 1$, the operator $\dibar$ on $\Lambda^r_{p,q}(\overline\Omega)$
is understood in the sense of distributions. 

\begin{proof}
For domains in $\C^n$, this is due to Shi and Yao 
\cite[Theorem~1.2 and Sect.~7]{ShiYao2025}. 
For smoothly bounded domains, Beals, Greiner and Stanton 
\cite[Theorem 2]{BealsGreinerStanton1987} proved this for the canonical 
solution operator to the $\dibar$-equation. In this case, the statement for a 
given $q\in\{1,\ldots,n\}$ holds if $\Omega$ has Stein interior and 
at each point of $b\Omega$ the Levi form has at least $n-q$ positive 
eigenvalues. The reduction to the Euclidean case follows the standard pattern. 
Properly embed $X$ in $\C^N$ with $N>n$ and identify it with its image. 
Choose a holomorphic tubular neighbourhood $V\subset \C^N$
of $X$ and a holomorphic retraction $\pi:V\to X$ by the 
Docquier--Grauert theorem (see \cite[Theorem 3.3.3]{Forstneric2017E}).
A standard construction gives a strongly pseudoconvex domain 
$\Omega_0 \Subset V$ with $\Cscr^2$ boundary such that 
$\Omega_0\cap X=\Omega$, $\pi(\overline \Omega_0)=\overline \Omega$, 
and $b\Omega_0$ intersects $X$ transversely along 
$b\Omega=b\Omega_0\cap X$.
Given $\alpha\in \Lambda^r_{p,q}(\overline\Omega)$ with $\dibar\alpha=0$,
we have $\pi^*\alpha \in \Lambda^r_{p,q}(\overline\Omega_0)$ and
$\dibar (\pi^*\alpha)=\pi^*(\dibar \alpha)=0$. 
Let $\widetilde T_q$ be the linear solution operator
on $\overline \Omega_0$ given by Shi and Yao \cite[Theorem~1.2]{ShiYao2025}. 
Let $\iota:\overline\Omega \hra \overline \Omega_0$ denote the natural 
inclusion. The operator $T_q$ on $\Omega$ defined by
$
 	T_q\alpha=\iota^*\widetilde T_q(\pi^*\alpha)
$
has the stated properties.
\end{proof}

\begin{corollary}\label{cor:dibar-little}
Write $\lambda^r_{p,q}$ for forms with coefficients of class $\lambda^r$ 
in smooth local frames. Under the hypotheses of Theorem \ref{th:dibar},
the solution operator $T_q$ restricts to a bounded operator
\[
 	T_q:\lambda^r_{p,q}(\overline\Omega)\cap\ker\overline\partial
	 \longrightarrow\lambda^{r+\alpha}_{p,q-1}(\overline\Omega),
	 \quad 0\le \alpha<1/2.
\]
\end{corollary}

\begin{proof}
Note that $\lambda^r_{p,q}(\overline\Omega)\cap\ker\overline\partial$
is a closed linear subspace of 
$\Lambda^r_{p,q}(\overline\Omega)\cap\ker\overline\partial$.
Combine Theorem \ref{th:dibar} with the continuous inclusion 
$\Lambda^{r+1/2}_{p,q-1}(\overline\Omega) 
\hookrightarrow \lambda^{r+\alpha}_{p,q-1}(\overline\Omega)$  
from Lemma \ref{lem:interpolation-dilation}.
\end{proof}

%
%
%
%
\section{Approximation of maps of class 
$\lambda^r_\Oscr$ on strongly pseudoconvex domains}
\label{sec:approximation}

In this section, we prove Theorem \ref{th:approximation} and its
parametric version, Theorem \ref{th:approximation-par}. 
We shall use the following notion of a Cartan pair; 
see \cite[Definitions 5.7.1 and 5.10.2]{Forstneric2017E}.

%
%
\begin{definition}\label{def:Cartan-pair}
A pair $(A,B)$ of compact sets in a complex manifold $X$ is a 
{\em Cartan pair} if 
\begin{enumerate}[\rm (i)]
\item $A$, $B$, $D=A\cup B$ and $C=A\cap B$ are closures of 
strongly pseudoconvex domains with $\Cscr^2$ boundary and Stein
interior, and
\item 
$\overline{A\setminus B}\cap \overline{B\setminus A} =\varnothing$.
\end{enumerate}
A Cartan pair $(A,B)$ is {\em special} if there is a coordinate
neighbourhood of $B$ in $X$ in which the sets $B$ and $C$ 
are strongly convex. In this case, $B$ is said to be a 
{\em convex bump} attached to $A$.
\end{definition}

A more general notion of a Cartan pair in
\cite[Definition 5.7.1 (I)]{Forstneric2017E} is suitable 
when considering holomorphic maps on neighbourhoods of 
the respective sets, which may be any Stein compacts. 
We shall not need it in this paper since our 
focus is on mapping spaces on domains with $\Cscr^2$ boundary.

\begin{proof}[Proof of Theorem \ref{th:approximation}]
Choose a finite convex-bump enlargement
\begin{equation}\label{eq:sequenceA}
 	A_0=\overline\Omega\subset A_1\subset\cdots\subset A_m,
 	\qquad\overline\Omega\subset\mathring A_m,
\end{equation}
as in \cite[Lemma 5.10.3]{Forstneric2017E}. For one step, let
$(A,B)$ be a special Cartan pair, $C=A\cap B$, and $D=A\cup B$. 
Choose convex coordinates near $B$, and translate them so that 
$0\in\mathring C$. First consider scalar functions.
For $f\in\Lambda^r_\Oscr(A)$, put 
\[ 
	f_t(x)=f(tx),\qquad x\in C,\quad 0<t<1. 
\] 
Since $tC\Subset\mathring C$, the function $f_t$ is holomorphic on a 
neighbourhood of $C$. Lemma \ref{lem:interpolation-dilation} gives 
\begin{equation}\label{eq:dilation}
	\lim_{t\uparrow1}\|f_t-f\|_{\Lambda^s(C)}=0
	\quad\text{for any $s\in (0,r)$}. 
\end{equation}
Fix $s\in (0,r)$, choose a number $\eta>0$, and then choose $t<1$ 
such that 
\[
	\|f_t-f\|_{\Lambda^s(C)}<\eta/2.
\]
Choose a compact convex neighbourhood $C_t$ of $C$ 
on which $f_t$ is holomorphic. Polynomial approximation of $f_t$ on $C_t$, 
followed by Cauchy estimates on $C$, gives a function $g$, holomorphic 
on a neighbourhood of $B$, such that 
\[ 
	\|f-g\|_{\Lambda^s(C)}<\eta. 
\] 
Choose a smooth function $\chi$ equal to $1$ near $\overline{A\setminus B}$ 
and to $0$ near $\overline{B\setminus A}$. Define 
\[ 
	u= \begin{cases} f,&\text{on }A\setminus B,\\ \chi f+(1-\chi)g,&\text{on }C,\\ 
	g,&\text{on }B\setminus A. \end{cases} 
\] 
The cutoff conditions make this a well-defined function of class 
$\Lambda^r(D)$, and 
\[ 
	\dibar u=(f-g)\dibar\chi 
\] 
on $C$, extended by zero to $D$. This extension has class $\Lambda^r$ 
and is $\dibar$-closed in $\mathring D$. Let $T$ be the solution 
operator on $\mathring D$ (see Theorem \ref{th:dibar}), and set 
\[ 
	\widetilde f=u-T(\dibar u). 
\] 
By Theorem \ref{th:dibar} we have $\widetilde f\in\Lambda^r_\Oscr(D)$ and 
\begin{equation}\label{eq:tildef} 
	\|\widetilde f-f\|_{\Lambda^s(A)} \le c\|f-g\|_{\Lambda^s(C)} <c\eta
\end{equation}
for a constant $c>0$ independent of $f$, $g$, and $\eta$. 
Although the $\Lambda^r$ norm of $\widetilde f$ need not remain bounded 
as $\eta\downarrow0$, each output still belongs to $\Lambda^r_\Oscr(D)$. 
Thus the same exponents $r$ and $s$ can be used at every step of the 
finite induction. If $f\in\lambda^r_\Oscr(A)$, \eqref{eq:dilation}
holds in $\Lambda^r(C)$, and the same argument applies with $s=r$. 
The cutoff construction preserves little-space regularity, and 
$T(\dibar u)\in\Lambda^{r+1/2}(D)\subset\lambda^r(D)$  
(see Corollary \ref{cor:dibar-little}). Hence 
$\widetilde f\in\lambda^r_\Oscr(D)$, and \eqref{eq:tildef} 
holds for $s=r$ (with a different constant $c>0$).

Iterate the one-step construction along \eqref{eq:sequenceA}, choosing the successive errors so that the sum of their $\Lambda^s(A_0)$ norms after restriction to $A_0$ is less than the prescribed tolerance. In the little-space case, use the same argument with $s=r$. The last function is holomorphic on $\mathring A_m$, which is a neighbourhood of $A_0=\overline\Omega$. This proves the scalar assertions.

For manifold-valued maps, let $W\subset X\times Y$ be a Stein 
neighbourhood of the graph of $f$; 
see \cite[Theorem 1.2]{Forstneric2007AJM}. Choose a proper 
holomorphic embedding $\iota:W\hookrightarrow\C^N$ and a holomorphic 
retraction $R:V\to\iota(W)$ from an open neighbourhood $V$ of $\iota(W)$. 
The map 
\[ 
	F(x)=\iota(x,f(x))\in \iota(W),\quad x\in\overline\Omega, 
\] 
has the same Zygmund regularity as $f$ and is holomorphic on $\Omega$. 
Apply the scalar result componentwise to approximate $F$ by a map $G$ 
holomorphic on a neighbourhood of $\overline\Omega$. If the approximation is 
sufficiently close uniformly, then, after shrinking that neighbourhood, 
$G$ takes values in $V$. The map 
\[ 
	g=\pr_Y\circ\iota^{-1}\circ R\circ G 
\] 
is holomorphic on a neighbourhood of $\overline\Omega$ and approximates $f$ 
in the required Zygmund topology. 

Conversely, maps holomorphic on a neighbourhood of $\overline\Omega$ are 
smooth there. Their limits in the $\Lambda^r$ topology belong locally to 
$\lambda^r$, and hence belong to $\lambda^r_\Oscr(\overline\Omega,Y)$.
\end{proof}

The following is a parametric version of Theorem \ref{th:approximation}. 

\begin{theorem}\label{th:approximation-par}
Let $\Omega,X,Y$ be as in Theorem \ref{th:approximation}, let
$P$ be a compact Hausdorff space, and let
$f:P\to\Lambda^r_\Oscr(\overline\Omega,Y)$ be continuous.
For every $0<s<r$, the maps $f(p)$ can be approximated uniformly in 
$p\in P$, in the $\Lambda^s(\overline\Omega,Y)$ topology,
by a continuous family of maps holomorphic on a common 
neighbourhood of $\overline\Omega$.
For a continuous family with values in $\lambda^r_\Oscr(\overline\Omega,Y)$,
approximation holds in the inherited $\Lambda^r$ norm.
For a relative version let $Q\subset P$ be closed and a strong
neighbourhood deformation retract. Assume that $f|_Q$ has
extensions forming a continuous map $Q\to\Oscr(U_Q,Y)$ for
one neighbourhood $U_Q$ of $\overline\Omega$, with the
compact-open topology. The approximants may then be chosen
equal to $f$ on $Q$, with extensions agreeing with the
prescribed ones on a smaller common neighbourhood of $\overline\Omega$.
\end{theorem}

\begin{proof} Put $K=\overline\Omega$. We use $a=s$ in the big-space case and $a=r$ in the little-space case. By Lemma~\ref{lem:interpolation-dilation}, in the former case the inclusion into $\lambda^s_\Oscr(K,Y)$ is continuous, as is seen in local coordinates. Thus in both cases we may start with a continuous map $f:P\to\lambda^a_\Oscr(K,Y)$ and approximate it in the inherited $\Lambda^a$ topology. Fix a metric $d$ inducing this topology; for example, one can use the metric induced by a smooth proper embedding of $Y$ into a Euclidean space and the corresponding $\Lambda^a$ norm. Let $\epsilon>0$. \smallskip\noindent 

\emph{Local families and interpolation.} For each $p_0\in P$, the graph of $f(p_0)$ has a Stein neighbourhood in $X\times Y$ by \cite[Theorem~1.2]{Forstneric2007AJM}. Hence we may choose a finite open cover $(V_i)$ of $P$ and Stein open sets $W_i\subset X\times Y$ such that the graph of $f(p)$ lies in $W_i$ for $p\in V_i$. As in the proof of Theorem~\ref{th:approximation}, choose a proper holomorphic embedding $\iota_i:W_i\hookrightarrow\C^{N_i}$ and a holomorphic retraction $R_i:O_i\to\iota_i(W_i)$ from an open neighbourhood $O_i$ of $\iota_i(W_i)$ in $\C^{N_i}$. Write \[ \Theta_i=\pr_Y\circ\iota_i^{-1}\circ R_i:O_i\longrightarrow Y. \] Let $L_i\subset V_i$ be compact. We claim that $f|_{L_i}$ can be approximated uniformly in $d$ by a jointly continuous family $g_i:L_i\times U_i\to Y$, holomorphic in $x\in U_i$ on a common neighbourhood $U_i$ of $K$. Indeed, the maps \[ F_{i,p}(x)=\iota_i(x,f(p)(x)), \qquad p\in L_i,\quad x\in K, \] form a compact family in $\lambda^a_\Oscr(K,\C^{N_i})$. Choose finitely many members of this family whose sufficiently small norm balls cover it. Approximate each chosen member by a map holomorphic on a neighbourhood of $K$, using the scalar case of Theorem~\ref{th:approximation}(ii) componentwise. A continuous partition of unity on $L_i$ subordinate to the inverse images of the norm balls gives a convex combination $G_i(p,x)$ of these finitely many approximants. It is defined on a common neighbourhood of $K$ and is arbitrarily close to $F_{i,p}$ in $\Lambda^a(K)$, uniformly for $p\in L_i$. By compactness, sufficiently close approximation puts $G_i(L_i\times K)$ in $O_i$. Shrinking the common neighbourhood to $U_i$, we have $G_i(L_i\times U_i)\subset O_i$. Then $g_i=\Theta_i\circ G_i$ has the asserted properties by continuity of superposition in the $\Lambda^a$ topology. For maps $u,v$ sufficiently close to $f(p)$, with $p\in L_i$, we use the interpolation \[ J_i(t,u,v)(x)= \Theta_i\bigl((1-t)\iota_i(x,u(x))+t\iota_i(x,v(x))\bigr), \qquad 0\le t\le 1. \] It is well defined on $K$ for a uniform closeness threshold over $L_i$: at $u=v=f(p)$ the arguments of $\Theta_i$ lie in a compact subset of $O_i$. This interpolation is continuous in the $\Lambda^a$ topology, has endpoints $u$ and $v$, and fixes the diagonal. In particular, given any positive error bound, the output lies within that bound of $f(p)$, uniformly in $p\in L_i$ and $t\in[0,1]$, if both inputs are sufficiently close to $f(p)$. If the two inputs are jointly continuous families holomorphic on neighbourhoods of $K$, the formula has the same properties on a common smaller neighbourhood: on a compact parameter set this follows by applying compactness to the open conditions defining the formula on the product with $K$. The endpoint identities also hold on that smaller neighbourhood. 

\smallskip\noindent \emph{Initial family in the relative case.} Suppose the relative data are given, and denote the prescribed extension by $\widehat f:Q\times U_Q\to Y$. Let $H$ be a strong deformation retraction of an open neighbourhood of $Q$ onto $Q$. Since $H$ fixes $Q$, compactness of $Q$ and $[0,1]$ allows us to shrink the parameter neighbourhood so that \[ d\bigl(f(H(p,t)),f(p)\bigr)<\epsilon/2 \qquad (0\le t\le 1) \] throughout it. Choose a continuous cutoff $\eta:P\to[0,1]$ supported in this smaller neighbourhood and equal to $1$ on a neighbourhood of $Q$. Set $\psi(p)=H(p,\eta(p))$ on the domain of $H$ and $\psi(p)=p$ elsewhere. Thus $f^0=f\circ\psi$ is within $\epsilon/2$ of $f$, and $\psi$ maps a neighbourhood of $Q$ into $Q$. Choose a closed neighbourhood $E_0$ of $Q$ within that neighbourhood. On $E_0$ we already have the desired family, given by $\widehat f(\psi(p),x)$ on $U_Q$, and it agrees with the prescribed extension on $Q$. In the nonrelative case put $E_0=\varnothing$ and $f^0=f$. \smallskip\noindent \emph{Finite patching.} Apply the local construction above to $f^0$ in place of $f$. Choose finitely many compact sets $K_i\subset\operatorname{int}L_i$, with $L_i\subset V_i$, 
such that the interiors of the $K_i$ cover 
$P\setminus\operatorname{int}E_0$. In the relative case the $V_i$ may be 
chosen disjoint from $Q$. Choose continuous 
$\chi_i:P\to[0,1]$ equal to $1$ on $K_i$ and with support in 
$\operatorname{int}L_i$. Put $E_i=E_{i-1}\cup K_i$. Inductively suppose there is a jointly continuous family $u$ on $E_{i-1}\times U$, holomorphic in $x$ on a common neighbourhood $U$ of $K$. On $E_{i-1}\cap L_i$ replace it by $J_i(\chi_i(p),u_p,g_{i,p})$. On $E_{i-1}\setminus\operatorname{int}L_i$ keep $u_p$, and on $K_i$ take $g_{i,p}$. The formula is defined on a common smaller neighbourhood of $K$, by the compactness observation above applied to $E_{i-1}\cap L_i$. The definitions agree on the overlaps: there $\chi_i$ is respectively $0$ or $1$. Pasting on these closed parameter sets gives the next jointly continuous family on $E_i$ with the required holomorphicity. In the relative case it is unchanged on $Q$. The errors can be chosen before carrying out the induction. Starting with a final tolerance of $\epsilon/2$ relative to $f^0$, choose the intermediate tolerances backwards using the uniform continuity at the diagonal of each interpolation on $L_i$. At step $i$ the preceding tolerance and the error of $g_i$ are taken small enough to ensure that the interpolation is defined and its output meets the next tolerance. On the parts where only $u$ or only $g_i$ is used, the same bound holds by choosing both input tolerances no larger than the output tolerance. This also covers the initial family on $E_0$, whose error relative to $f^0$ is zero. After finitely many steps $E_i=P$. The resulting family is jointly continuous on $P\times U'$ and holomorphic in $x$, for one neighbourhood $U'$ of $K$. Joint continuity implies continuity into $\Oscr(U',Y)$ with the compact-open topology. Its restriction to $K$ is within $\epsilon$ of $f$, uniformly on $P$. In the relative case, $U'$ can be taken inside $U_Q$, and the family agrees there with $\widehat f$ for every parameter in $Q$. \end{proof}

%
%
\section{Cartan's lemma and Theorem A for vector bundles of class 
$\lambda^r_\Oscr$}
\label{sec:vectorbundles}

The notion of a vector bundle of Zygmund class 
$\lambda^r_\Oscr(\overline\Omega)$, $r>0$, was introduced in 
Subsect.\ \ref{ss:Lr-bundles}. In this section, we prove the following version 
of Cartan's Theorem A for such bundles. 

%
%
\begin{theorem}\label{th:TheoremA}
Let $\overline \Omega$ be a compact, strongly pseudoconvex domain with $\Cscr^2$ boundary in a Stein manifold $X$,  
and let $\pi:E\to \overline \Omega$ be a vector bundle 
of class $\lambda^r_\Oscr(\overline\Omega)$ for some $r>0$. 
There exist finitely many sections in $\Gamma^r_\Oscr(\overline\Omega,E)$ 
spanning every fibre $E_x:=\pi^{-1}(x)$ for $x\in \overline \Omega$.
\end{theorem}

The classical Theorem A of Henry Cartan gives such a statement 
for holomorphic vector bundles on finite dimensional Stein spaces.
For vector bundles of class $\Ascr^r(\overline\Omega)$ with $r\in\Z_+$,
and for more general coherent analytic sheaves of this class,
the analogue of Theorem \ref{th:TheoremA} is due to 
Leiterer \cite{Leiterer1976II} for domains in Euclidean spaces and
to Heunemann \cite[Theorems 2, 6]{Heunemann1986II} for domains
in Stein manifolds.

\begin{remark}\label{rem:TheoremA}
An equivalent statement of Theorem \ref{th:TheoremA} 
is that every vector bundle $E\to \overline \Omega$ as in 
the theorem admits a vector bundle epimorphism
$\Phi:\overline\Omega\times \C^m \to E$ of class
$\lambda^r_\Oscr(\overline\Omega)$ for some $m\in\N$.
Indeed, if $\xi_1,\ldots,\xi_m:\overline\Omega\to E$
are sections of class $\lambda^r_\Oscr(\overline\Omega)$ 
spanning every fibre of $E$, then the map 
$\Phi:\overline\Omega\times \C^m \to E$ given by 
\begin{equation}\label{eq:Phi}
	\Phi(x,z_1,\ldots,z_m)= \sum_{i=1}^m z_i\xi_i(x) \in E_x=\pi^{-1}(x),
	\quad x\in\overline\Omega,\ z=(z_1,\ldots,z_m)\in\C^m
\end{equation}
is a vector bundle epimorphism of class 
$\lambda^r_\Oscr(\overline\Omega)$. Conversely,
given such an epimorphism, the images of standard basis sections
of the trivial bundle $\overline\Omega\times \C^m$ generate each fibre of $E$. 
\end{remark}

We also have the following embedding result for vector bundles
of class $\lambda^r_\Oscr(\overline\Omega)$. 

\begin{corollary}\label{cor:embedding}
Given a vector bundle $\pi:E\to \overline \Omega$ of class
$\lambda^r_\Oscr(\overline\Omega)$ as in Theorem \ref{th:TheoremA},
there exists a vector bundle embedding
$E\hookrightarrow \overline \Omega\times \C^m$ of class
$\lambda^r_\Oscr(\overline\Omega)$ for some $m\in\N$.
\end{corollary}

\begin{proof}
Theorem \ref{th:TheoremA} gives a vector bundle epimorphism
$\Phi: \overline \Omega \times \C^m \to E^*$ of class
$\lambda^r_\Oscr(\overline\Omega)$. Its dual $\Phi^*:(E^*)^*=E \to 
(\overline \Omega \times \C^m)^* \cong \overline \Omega \times \C^m$
is a vector bundle embedding of class 
$\lambda^r_\Oscr(\overline\Omega)$.
\end{proof}
\vspace{2mm}

Denote by $M_n=M_n(\C)\cong \C^{n\times n}$ 
the space of complex $n\times n$ matrices
and note that $GL_n=GL_n(\C)\subset M_n$. 
By $I\in GL_n$ we denote the identity matrix. 
In the proof of Theorem \ref{th:TheoremA} we shall use the following
version of Cartan's lemma. See also the parametric version,
Lemma \ref{lem:Cartan-par}, and \cite{Cartan1940}
or \cite[Section VI.\ E]{GunningRossi1965} for the classical 
version of Cartan's lemma.

%
%
\begin{lemma}\label{lem:Cartan}
Let $(A,B)$ be a Cartan pair in a complex manifold $X$
(see Definition \ref{def:Cartan-pair}) such that $C:=A\cap B$
is holomorphically convex in $B$.
Given a map $\gamma\in \lambda^r_\Oscr(C,GL_n)$ 
which is homotopic to the constant map $C\ni x\mapsto I\in GL_n$ 
by a path in $\Cscr(C,GL_n)$ and a number $\epsilon>0$, there are 
maps $\alpha \in \lambda^r_\Oscr(A,GL_n)$ and 
$\beta \in \lambda^r_\Oscr(B,GL_n)$ such that 
$\|\alpha-I\|_{\lambda^r(A)} <\epsilon$ and 
\begin{equation}\label{eq:split}
	\gamma = \alpha^{-1} \cdotp \beta\ \ \ \text{holds on $C$}.
\end{equation}
\end{lemma}

\begin{proof}
We first construct the product splitting \eqref{eq:split}
for $\gamma$ close to $I\in GL_n$. Write 
$\gamma=I+c$ with $c\in \lambda^r_\Oscr(C,M_n)$.
We have the following solution to the Cousin problem 
in $\lambda^r_\Oscr$.

%
%
\begin{lemma}\label{lem:Cousin}
For every $r>0$ there are bounded linear operators 
$\Acal:\lambda^r_\Oscr(C)\to \lambda^r_\Oscr(A)$ and 
$\Bcal:\lambda^r_\Oscr(C)\to \lambda^r_\Oscr(B)$ with
\begin{equation}\label{eq:AplusB}
	\Acal c - \Bcal c =c\ \ \ \text{for all}\ c\in \lambda^r_\Oscr(C).
\end{equation}
The analogous statement holds in big Zygmund spaces $\Lambda^r$.
\end{lemma}

\begin{proof}
Set $D=A\cup B$. Condition (ii) in Definition \ref{def:Cartan-pair}
implies that there is a smooth function $\chi: X \to[0,1]$ 
which equals $0$ on a neighbourhood of $\overline {A\setminus B}$ 
and equals $1$ on a neighbourhood of $\overline {B \setminus A}$. 
Given $c\in \lambda^r_\Oscr(C)$, we have $c\chi \in \lambda^r(A)$,
$c(1-\chi)\in \lambda^r(B)$, and 
$\dibar (c \chi)=c\, \dibar \chi \in \lambda^r_{0,1}(D)$ 
is a closed $(0,1)$-form on $D$ of class $\lambda^r(D)$ 
supported on $C$. Let 
\[
	T=T_1: \{\omega\in \lambda^r_{0,1}(D):\dibar \omega=0\}
	\to \lambda^r(D)
\]
denote the bounded linear solution operator 
for the $\dibar$-equation for $(0,1)$-forms of class
$\lambda^r(D)$; see Theorem \ref{th:dibar}. 
Cutoff terms are little, and the solution term has regularity
$r+1/2$, hence is little at order $r$ by
Corollary \ref{cor:dibar-little}.
The linear operators on $\lambda^r_\Oscr(C)$ defined by 
\[
	\Acal c= c \, \chi - T(c\, \dibar \chi)\in \lambda^r_\Oscr(A), \quad 
	\Bcal c= c(\chi-1) - T(c\, \dibar \chi) \in \lambda^r_\Oscr(B)
\] 
for all $c\in \lambda^r_\Oscr(C)$ then satisfy the lemma.
\end{proof}
\vspace{2mm}

Applying Lemma \ref{lem:Cousin} componentwise gives
bounded linear operators 
\[
	\Acal:\lambda^r_\Oscr(C,M_n) \to \lambda^r_\Oscr(A,M_n),
	\qquad 
	\Bcal:\lambda^r_\Oscr(C,M_n) \to \lambda^r_\Oscr(B,M_n)
\]
satisfying $\Acal-\Bcal=\Id$ on $\lambda^r_\Oscr(C,M_n)$;
see \eqref{eq:AplusB}. We define a map 
$\Phi:U\to \lambda^r_\Oscr(C,GL_n)$ on a small neighbourhood 
$U\subset \lambda^r_\Oscr(C,M_n)$ of $c=0$ by
\[
	\Phi(c) = (I-\Acal c)^{-1} (I-\Bcal c)
	\in \lambda^r_\Oscr(C,GL_n), \quad c\in U.
\]
Clearly, $\Phi$ is smooth and satisfies $\Phi(0)=I$ 
and $d\Phi_0 = \Acal - \Bcal =\Id$. It follows that 
$\Phi$ has a smooth right inverse $\Psi$ on a neighbourhood 
$V\subset \lambda^r_\Oscr(C,GL_n)$ of $I$ such that $\Psi(I)=0$ and 
\[
	(\Phi\circ \Psi)(\gamma) = 
	\left(I-\Acal \circ \Psi(\gamma)\right)^{-1}  
	\left(I-\Bcal \circ\Psi(\gamma)\right) = \gamma 
	\quad \text{for all}\ \gamma\in V.
\]
The smooth operators $\wt \Acal:V\to \lambda^r_\Oscr(A,GL_n)$,
$\wt \Bcal:V\to \lambda^r_\Oscr(B,GL_n)$ defined by
\[
	\wt \Acal = I - \Acal \circ\Psi,\qquad 
	\wt \Bcal = I - \Bcal \circ\Psi
\] 
then provide a splitting 
$\gamma=(\wt \Acal \gamma)^{-1} (\wt \Bcal \gamma)$ 
for $\gamma\in V$ (see \eqref{eq:split}) satisfying 
\begin{equation}\label{eq:estAB}
	\|\wt \Acal \gamma-I\|_{\lambda^r(A)} \le const 
	\|\gamma-I\|_{\lambda^r(C)},
	\quad 
	\|\wt \Bcal \gamma-I\|_{\lambda^r(B)} \le const 
	\|\gamma-I\|_{\lambda^r(C)}
\end{equation}
for some $const>0$ depending only on $(A,B)$ and $r$.
Taking $\alpha=\wt \Acal \gamma$ and $\beta=\wt \Bcal \gamma$
proves the lemma for maps $\gamma\in \lambda^r_\Oscr(C,GL_n)$
near the constant map $C\ni x\mapsto I$. Note that this proof
also applies in big Zygmund spaces. 

The general case is obtained as follows.
By Theorem \ref{th:approximation} we can approximate 
$\gamma$ as closely as desired in $\lambda^r_\Oscr(C)$ by 
a holomorphic map ${\gamma\,}' :U\to GL_n$ from an 
open neighbourhood $U\subset X$ of $C$. 
(This step fails in big Zygmund spaces unless we decrease the
Zygmund exponent.) Since $GL_n$ is an Oka manifold 
(see Grauert \cite{Grauert1957I,Grauert1957II} or 
\cite[Proposition 5.6.1]{Forstneric2017E}), 
$\gamma$ is homotopic to the constant map,
and $C$ is holomorphically convex in $B$, 
the Oka principle \cite[Corollary 5.4.5]{Forstneric2017E}
shows that ${\gamma\,}'$ can be approximated uniformly
on a compact neighbourhood $C'\subset U$ of $C$ by a holomorphic
map $\tilde \gamma\in \Oscr(B,GL_n)$. 
Then, $\gamma= (\gamma {\tilde \gamma}^{-1}) \tilde\gamma$
on $C$, and $\gamma {\tilde \gamma}^{-1}$ is close to $I$ 
in $\lambda^r(C,GL_n)$. By the first part, we have
$\gamma {\tilde \gamma}^{-1}=\alpha^{-1}\tilde \beta$
with $\alpha \in \lambda^r_\Oscr(A,GL_n)$ close to $I$ and 
$\tilde \beta \in \lambda^r_\Oscr(B,GL_n)$.
Setting $\beta=\tilde\beta \tilde\gamma\in \lambda^r_\Oscr(B,GL_n)$ 
gives $\gamma=\alpha^{-1}\beta$ as in \eqref{eq:split}.
\end{proof}

%
%
Lemma \ref{lem:Cartan} has the following generalisation
to the parametric case.

\begin{lemma}\label{lem:Cartan-par}
Assume that $X$ and $(A,B)$ are as in Lemma  \ref{lem:Cartan}.
Given a compact Hausdorff space $P$, a closed subspace 
$Q\subset P$ which is a strong neighbourhood deformation retract, 
and a continuous map 
$\gamma : P\to \lambda^r_\Oscr(C,GL_n)$ such that $\gamma(p)=I$
for all $p\in Q$ and there is a homotopy 
$\gamma_t:P\to \Cscr(C,GL_n)$ $(t\in [0,1])$ which is fixed 
on $Q$ such that $\gamma_0=I$ and $\gamma_1=\gamma$,
there are continuous maps 
\[
	\alpha:P \to \lambda^r_\Oscr(A,GL_n),\qquad 
	\beta:P\to  \lambda^r_\Oscr(B,GL_n)
\]
such that $\|\alpha-I\|_{\lambda^r(A)}$ is arbitrarily small,
$\alpha|_Q=\beta|_Q=I$, and 
$\gamma = \alpha^{-1} \cdotp \beta$ holds on $C$.
\end{lemma}

\begin{proof}
Use approximation in Theorem~\ref{th:approximation-par}
to obtain $\gamma'_p$ on a common neighbourhood of $C$, with
$\gamma'_q=I$ for $q\in Q$. Use the homotopy and cutoff construction
above, stationary on $Q$. Relative Oka approximation gives
$\tilde\gamma_q=I$ for $q\in Q$. Cauchy estimates give uniform
$\Lambda^r(C)$ smallness of $\gamma_p\tilde\gamma_p^{-1}-I$.
Apply the continuous near-identity splitting and multiply the second
factors by $\tilde\gamma_p$. The near-identity operators send $I$ to
$(I,I)$, so the factors on $Q$ are unchanged.
\end{proof}

\begin{remark}\label{rem:Liegroup}
Lemmas \ref{lem:Cartan} and \ref{lem:Cartan-par} hold for any
complex Lie group $G$. In an exponential chart on its Lie
algebra apply the Cousin operators componentwise. The map
\[
 c\longmapsto\exp(-\Acal c)^{-1}\exp(-\Bcal c)
\]
has derivative $\Acal-\Bcal=\Id$ at $0$ in logarithmic
coordinates. The Banach inverse function theorem gives the
little Zygmund space factorisation. For a general null-homotopic map 
use little-space approximation and the Oka principle for $G$, as
in the matrix case. These operations preserve little spaces
and continuous parameter dependence. 
\end{remark}

The main step in the proof of Theorem \ref{th:TheoremA} is 
given by the following lemma.

%
%
\begin{lemma}\label{lem:ext-to-bump}
Assume that $(A,B)$ is a special Cartan pair 
(see Definition \ref{def:Cartan-pair}).
Let $\pi:E\to D=A\cup B$ be a vector bundle of class
$\lambda_\Oscr^r(D)$, $r>0$, such that $E|_B$ is 
$\lambda_\Oscr^r(B)$-isomorphic to a trivial bundle.
Then, every vector bundle epimorphism 
$\Phi:A\times\C^n \to E|_A$ of class $\lambda_\Oscr^r(A)$
can be approximated in $\lambda^r(A)$ by vector bundle 
epimorphisms $\wt \Phi:D\times\C^n \to E$ of class $\lambda_\Oscr^r(D)$. 
\end{lemma}

\begin{proof}
Let $e_1,\ldots,e_n$ be sections of $E|_A$ of class $\lambda^r_\Oscr(A)$
obtained by applying $\Phi$ to the standard basis sections
of $A\times\C^n$. Also, let $g_1,\ldots,g_m$ with $m=\rank E$
be basis sections of class $\lambda^r_\Oscr(B)$ 
of the trivial bundle $E|_B\cong B\times\C^m$. On $C$, we have 
$e_i=\sum_{j=1}^m \gamma_{i,j}g_j$ ($i=1,\ldots,n$) 
with $\gamma_{i,j}\in\lambda^r_\Oscr(C)$, and 
the $n\times m$ matrix $\Gamma'(x)=(\gamma_{i,j}(x))$ 
has rank $m$ at every point $x\in C$. We claim that there is
an $n\times n$ matrix function 
\[
	\Gamma=(\gamma_{i,j})_{i,j=1}^n =(\Gamma',\Gamma'') 
	\in \lambda^r_\Oscr(C,GL_n)
\]
whose first $m$ columns form $\Gamma'$.
This is equivalent to saying that the subbundle $E'$ of  
$C\times\C^n$, spanned by the columns of $\Gamma'$, has a trivial complementary subbundle $E''\subset C\times\C^n$ of class 
$\lambda^r_\Oscr(C)$. The bundle of frames of complements of $E'$ has fibre the product
of a complex affine space and $GL_{n-m}(\C)$. This follows from
the local graph description in frames adapted to $E'$. The fibre
is Oka, and the bundle has a continuous section because the
orthogonal complement is topologically trivial on $C$. The
$\Ascr^0$ case of \cite[Theorem~1.7]{DrinovecForstneric2008FM}
gives a continuous frame, holomorphic in the interior, of a
complementary subbundle $\wt E''$.
Let $\wt \Gamma''$ be an $n\times (n-m)$ matrix of class $\Ascr(C)$
whose columns are the basis sections of $\wt E''$. 
Since $C$ is convex, we can approximate $\wt \Gamma''$ uniformly on $C$ 
by a matrix $\Gamma''$ of class $\lambda^r_\Oscr(C)$. 
If the approximation is close enough then 
$\Gamma=(\Gamma',\Gamma'') \in \lambda^r_\Oscr(C,GL_n)$. 
Set $g_{m+1}=0,\ldots,g_n=0$. It follows that 
\[
	(e_1,e_2,\ldots,e_n)^t = \Gamma (g_1,g_2,\ldots,g_n)^t,  
\]
where the superscript $t$ denotes the transpose. By 
Lemma \ref{lem:Cartan} we have $\Gamma=\alpha^{-1}\beta$
where $\alpha\in \lambda^r_\Oscr(A,GL_n)$ is close to $I$
and $\beta \in \lambda^r_\Oscr(B,GL_n)$. Then,
\[
	\alpha (e_1,e_2,\ldots,e_n)^t = \beta (g_1,g_2,\ldots,g_n)^t
\]
holds on $C$. The resulting family of $n$ sections of $E\to D$ 
of class $\lambda^r_\Oscr(D)$ determines a vector bundle epimorphism 
$\wt \Phi:D\times \C^n\to E$ of class $\lambda^r_\Oscr(D)$ 
(see \eqref{eq:Phi}) approximating $\Phi$ in $\lambda^r_\Oscr(A)$.
\end{proof}

\begin{proof}[Proof of Theorem \ref{th:TheoremA}]
Pick a $\Cscr^2$ strongly plurisubharmonic function $\rho$ on a 
neighbourhood $U$ of $\overline\Omega$
such that $\Omega=\{x\in U:\rho(x)<0\}$ and $d\rho_x\ne 0$
for every point $x\in b\Omega =\{\rho=0\}$. Choose $c<0$ such that
$\rho$ has no critical values on $[c,0]$ and 
set $A_0=\{\rho\le c\}$. By \cite[Lemma 5.10.3]{Forstneric2017E}, 
given an open cover $\Ucal=\{U_j\}$ of 
$\overline{\Omega\setminus A_0}=\{c\le \rho\le 0\}$ consisting of
holomorphic coordinate charts $U_j\subset X$,   
there are compact, strongly pseudoconvex domains with $\Cscr^2$ boundary 
$A_0\subset A_1 \subset \cdots \subset A_{k_0}=\overline\Omega$
for some $k_0\in\N$ such that for every $k=0,1,\ldots,k_0-1$ we have 
$A_{k+1}=A_k\cup B_k$, where $(A_k,B_k)$ is a special Cartan pair
(see Definition \ref{def:Cartan-pair}) and $B_k\subset U_j$
for some $j=j(k)$. We choose the sets $U_j$ small enough so that 
the restricted bundle $E|_{B_k}$ is trivial for $k=0,1,\ldots,k_0-1$.
Since $E$ is holomorphic over $\Omega$,
there is a holomorphic vector bundle epimorphism
$\Phi_0:A_0\times \C^n \to E|_{A_0}$ for some $n\in \N$. 
We inductively apply Lemma \ref{lem:ext-to-bump}
to find vector bundle epimorphisms
$\Phi_k:A_k\times \C^n \to E|_{A_k}$ of class $\lambda^r_\Oscr(A_k)$
for $k=1,\ldots,k_0$ such that $\Phi_{k+1}$ approximates $\Phi_k$ 
on $A_k$ for every $k=0,1,\ldots,k_0-1$.
Then, $\Phi:=\Phi_{k_0}:\overline \Omega \times \C^n\to E$ 
satisfies the conclusion of the theorem.
\end{proof}

Next, we show that every complex vector subbundle 
of class $\lambda^r_\Oscr(\overline\Omega)$ is complemented.

%
%
\begin{theorem}\label{th:complement}
Let $\Omega$ be a relatively compact strongly pseudoconvex domain with $\Cscr^2$ boundary in a Stein manifold, and let
$\pi:E\to \overline\Omega$ be a complex vector bundle of class
$\lambda^r_\Oscr(\overline\Omega)$. For every 
complex vector subbundle $E'\subset E$ of class 
$\lambda^r_\Oscr(\overline\Omega)$ there is a complex vector subbundle $E''\subset E$ complementary to $E'$ of class 
$\lambda^r_\Oscr(\overline\Omega)$ such that $E\cong E'\oplus E''$.
\end{theorem}

\begin{proof}
We first consider the case when $E= \overline\Omega\times\C^n$
is a trivial bundle. The complements of $E'$ form a fibre bundle
with affine-space fibres, described locally by graphs in frames
adapted to $E'$. Orthogonal complements give a continuous section.
The $\Ascr^0$ case of \cite[Theorem~1.7]{DrinovecForstneric2008FM}
gives a complementary subbundle $\wt E$, continuous on the closure
and holomorphic on the interior.
To complete the proof, it suffices to approximate $\wt E$ sufficiently closely 
by a subbundle $E''\subset E=\overline\Omega\times\C^n$ of class 
$\lambda^r_\Oscr(\overline\Omega)$. We follow the idea
in \cite[Proof of Lemma 2.2]{ForstnericLowOvrelid2001}.

Let $\wt L: \overline \Omega \to M_n=\Lin_\C(\C^n,\C^n)$ be the unique map
such that $\wt L(x):\C^n\to\C^n$ is the projection onto $\wt E_x$ 
with kernel $E'_x$ for every $x\in \overline \Omega$.
Clearly, $\wt L$ is of class $\Ascr(\overline\Omega)$, and 
the spectrum of $\wt L(x)$ is contained in $\{0,1\}$ for every 
$x\in \overline\Omega$. Let $C=\{z\in \C : |z-1|=1/2\}$. 
By \cite[Theorem~1.2]{DrinovecForstneric2008FM} at exponent zero, we can approximate $\wt L$ uniformly on $\overline\Omega$ 
by a holomorphic map $L:U \to M_n$ 
on a neighbourhood $U$ of $\overline\Omega$. 
Assuming that the approximation is close enough, 
$L(x)$ has no eigenvalues on the curve $C$ for $x\in\overline\Omega$.
We then have $\C^n=V_{x,+} \oplus V_{x,-}$
where $V_{x,+}$ resp.\ $V_{x,-}$ are the $L(x)$-invariant 
subspaces of $\C^n$ spanned by the generalised
eigenvectors of $L(x)$ inside resp.\ outside of $C$, and  
$V_{x,+}$ is close to the subspace $\wt E_x=\wt L(x)(\C^n) \subset\C^n$. 
The map 
\[
	P(L(x))= \frac{1}{2\pi \imath} 
	\int_{\zeta\in C} \left(\zeta I-L(x) \right)^{-1}\, d\zeta \in M_n
\]
is the projection of $\C^n$ onto $V_{x,+}$ with kernel $V_{x,-}$ 
(see \cite{GohbergLancasterRodman1986}).
It depends holomorphically on $x$ in a neighbourhood of $\overline\Omega$,
and it approximates $\wt L(x)$ uniformly for $x\in \overline\Omega$. 
Assuming that the approximations are close enough, the subbundle 
$E''\subset E$ with fibres $E''_x=V_{x,+}=P(L(x))(\C^n)$, 
$x\in \overline\Omega$, is complementary to $E'$, 
and $E''$ is holomorphic on a neighbourhood of $\overline\Omega$. 

For a general vector bundle $E\to\overline\Omega$
of class $\lambda^r_\Oscr(\overline\Omega)$,  
Theorem \ref{th:TheoremA} gives a vector bundle
epimorphism $\Phi:\overline\Omega\times\C^N\to E$ of class
$\lambda^r_\Oscr(\overline\Omega)$. The preimage $\Phi^{-1}(E')$
is a vector subbundle of $\overline\Omega\times\C^N$ of 
class $\lambda^r_\Oscr(\overline\Omega)$, 
so it has a complementary subbundle 
$\wt E''\subset \overline\Omega\times\C^N$
of class $\lambda^r_\Oscr(\overline\Omega)$. 
Its image $E''=\Phi(\wt E'')$ is then a subbundle of $E$
of class $\lambda^r_\Oscr(\overline\Omega)$ which is 
complementary to $E'$. 
\end{proof}

%
%
\begin{remark}\label{rem:stability}
By a continuous family of vector bundles of class 
$\lambda^r_\Oscr(\overline\Omega)$ 
we mean that, locally in the parameter, the bundles are represented on a 
common finite trivialising cover by transition matrices depending continuously 
on the parameter in the local $\Lambda^r$ norms on compact subsets of the 
overlaps. After shrinking the cover to a finite compact refinement whose 
relative interiors still cover $\overline\Omega$, we may consequently work 
with a continuous family of cocycles in the Banach norms on the refined 
overlaps. This is the little H\"older--Zygmund analogue of the setup for the 
stability theorem in \cite[Theorem~2.7]{Leiterer1990}. 
\end{remark} 

\begin{theorem}\label{th:stability} 
Let $\overline\Omega$ be a compact, strongly pseudoconvex domain with 
$\Cscr^2$ boundary in a Stein manifold $X$, let $r>0$, and let $P$ be a 
topological space. If $E_p\to \overline\Omega$, $p\in P$, is a continuous 
family of vector bundles of class $\lambda^r_\Oscr(\overline\Omega)$, 
then every $p_0\in P$ has a neighbourhood $P_0$ on which there is a 
continuous family of vector bundle isomorphisms 
\[ 
	\phi_p:E_p\longrightarrow E_{p_0},\qquad p\in P_0, 
\] 
of class $\lambda^r_\Oscr(\overline\Omega)$, with $\phi_{p_0}=\Id$. \end{theorem} 

\begin{proof} Fix $p_0\in P$, put $E=E_{p_0}$, and set 
$F=\operatorname{End}(E)$. Choose a common finite compact trivialising 
refinement $\mathcal K=(K_i)_{i=1}^m$ whose relative interiors cover 
$\overline\Omega$. Each $K_i$ is the closure of a strongly pseudoconvex 
$\Cscr^2$ domain in a trivialising set. At an interior point take a small ball. 
At a boundary point use holomorphic convexifying coordinates, intersect
with a small transverse coordinate ball centred at that point and
round the corner by a regularised maximum, leaving a relative
neighbourhood of its centre unchanged. Compactness gives a finite
refinement. No smoothness of the intersections is required.

\smallskip\noindent \emph{Bounded solutions of the additive Cousin equations.} Write $C^j$ for the Banach space of alternating \v Cech $j$-cochains on 
$\mathcal K$ with values in little-class holomorphic sections of $F$, and let 
$\delta$ be the \v Cech differential. For a multi-index $I$ put 
$U_I=\bigcap_{i\in I}\mathring K_i$ and
$K_I=\overline{U_I}$ in $X$. Use $K_I$ rather than raw intersections,
omit empty intersections, and impose holomorphicity on $U_I$.
Identities between continuous cochains extend by density. All
relative-open overlaps in $\overline\Omega$ lie in the corresponding $K_I$.
On these compact intersections use restriction H\"older--Zygmund norms 
in local coordinates, with the little spaces defined by closure of smooth 
restrictions and holomorphicity imposed on the interior. Equivalently, for 
$F$-valued cochains one may use the norm of their components in a fixed embedding of $F$ as a direct summand of a trivial bundle. The cochain 
norms are the maxima of the finitely many component norms. 
Smooth cutoffs, restrictions and multiplication by fixed little-class bundle morphisms are bounded in these norms. Choose a smooth partition of unity 
$(\chi_i)$ on $\overline\Omega$ subordinate to the relative interiors of 
the sets $K_i$. The usual partition-of-unity operator on cochains of sections 
or forms is 
\[ 	(hc)_{i_0\ldots i_{j-1}} 
	=\sum_{\nu=1}^m\chi_\nu c_{\nu i_0\ldots i_{j-1}}, \qquad j\ge1 . 
\] 
Each term is extended by zero off $K_\nu$. These extensions are bounded in 
the little norms because the cutoffs are supported in the relative interiors of the 
corresponding sets. We have $\delta h+h\delta=\Id$ in positive degrees. The 
operator $h$ does not preserve holomorphicity. We first record the solution 
operator for $F$-valued forms on $\overline\Omega$. 
By Corollary~\ref{cor:embedding} and Theorem~\ref{th:complement}, 
identify $F$ with a direct summand of a trivial bundle and let $\Pi$ be the 
corresponding projection of class $\lambda^r_\Oscr(\overline\Omega)$. 
Let $T_q$ denote the solution operator from Theorem \ref{th:dibar}. 
Viewing an $F$-valued $\dibar$-closed $(p,q)$-form $\omega$ as a form with values in the trivial bundle, set 
\[ 
	S_q\omega=\Pi T_q\omega , 
\] 
where $T_q$ is applied componentwise. The inclusion and $\Pi$ are 
holomorphic on $\Omega$, so $\dibar S_q\omega=\omega$. By 
Corollary~\ref{cor:dibar-little} and the Banach algebra estimate, 
$S_q$ is bounded in the little norm at exponent $r$. On each $K_i$, 
the analogous bounded solution operator for $F$-valued forms is obtained 
by applying Corollary~\ref{cor:dibar-little} in a trivialisation of $F$. 
All $\dibar$-identities here are taken on the interiors. Let 
$Z^j=\ker(\delta:C^j\to C^{j+1})$. We construct bounded linear maps 
$R_j:Z^j\to C^{j-1}$ such that $\delta R_j=\Id$ for $j=1,2$. For $c\in Z^1$, 
put $b=hc$. Then $\delta b=c$, and the forms $\dibar b_i$ agree on the 
overlaps and define a global closed $(0,1)$-form $\omega$ 
on $\overline\Omega$. It has little coefficients since differentiation of $b_i$ 
in holomorphic frames falls only on the smooth cutoffs. Hence 
\[ 
	(R_1c)_i=b_i-(S_1\omega)|_{K_i} 
\] 
is holomorphic on the interior of $K_i$, has little coefficients, and satisfies 
$\delta R_1c=c$. For $c\in Z^2$, put $b=hc$ and $d=\dibar b$. 
Then $\delta b=c$, $\delta d=0$, and $\dibar d=0$. Put $a=hd$, so 
$\delta a=d$. The forms $\dibar a_i$ agree on overlaps and define a global 
closed $(0,2)$-form $\eta$ with little coefficients. Indeed, in computing 
$\dibar a_i$ only the cutoffs are differentiated, since $\dibar d=0$. 
Set $u=S_2\eta$; if $\dim X=1$, take $u=0$. The forms $a_i-u|_{K_i}$ are 
closed. Solve 
\[ 
	\dibar v_i=a_i-u|_{K_i} 
\] 
on each $K_i$ using the local solution operator. It follows that $R_2c=b-\delta v$ 
is holomorphic on the interiors of the overlaps and satisfies $\delta R_2c=c$. 
All operations in these constructions are bounded at exponent $r$. In particular, 
$\delta:C^0\to C^1$ has image $Z^1$, and $\delta:C^1\to C^2$ has closed 
image $Z^2$. 

\smallskip\noindent \emph{Stability of nearby cocycles.} Let $A_{ij}$ be the 
transition matrices of $E$ and $B_{ij}$ those of a nearby bundle on the 
same refinement. For $i<j$ write 
\[ 
	B_{ij}=(I+e_{ij})A_{ij}. 
\] 
Regard $e_{ij}$ as an endomorphism in frame $i$, and use alternation to 
regard the entries for $i<j$ as a cochain $e\in C^1$. On triples $i<j<k$, 
the cocycle equations give 
\[ 
	(\delta e)_{ijk} =-e_{ij} A_{ij}e_{jk}A_{ji}. 
\] 
Thus $\|\delta e\|\le C\|e\|^2$. Set 
\[ 
	e^\circ=e-R_2(\delta e),\qquad a=R_1e^\circ . 
\] 
These expressions are defined since $\delta e\in Z^2$ and $e^\circ\in Z^1$. 
We have $\|a\|\le C_1\|e\|$ for $e$ small, and $\delta a=e^\circ$. For $e$ 
sufficiently small the matrices $I+a_i$ are invertible. Make the change 
\[ 
	B'_{ij}=(I+a_i)B_{ij}(I+a_j)^{-1}, \qquad B'_{ij}=(I+e'_{ij})A_{ij}. 
\] 
In frame $i$, writing $\widetilde a_j=A_{ij}a_jA_{ji}$, we have 
\[ 
	e'_{ij} =\bigl(e_{ij}-(\delta a)_{ij}+a_i e_{ij}\bigr) (I+\widetilde a_j)^{-1}. 
\] 
Since $e-\delta a=R_2(\delta e)$, the Banach algebra estimate gives 
$\|e'\|\le C_2\|e\|^2$, with a constant independent of the nearby cocycle. 
Iterate this change of frame. On a sufficiently small fixed neighbourhood 
of $A$, the errors decrease at least geometrically and the norms of the 
successive corrections $a_i$ are summable, uniformly over the initial cocycles 
in that neighbourhood. Hence the accumulated products and their inverses
converge in the little norms to matrices $G_i$ and $G_i^{-1}$. Each finite product 
depends continuously on the initial cocycle, and the convergence is uniform there, 
so the limits depend continuously on it. Passing to the limit gives 
\[ 
	G_i B_{ij}=A_{ij}G_j . 
\] 
The matrices $G_i$ therefore represent an isomorphism from the bundle with transitions $B_{ij}$ to $E$. When $B=A$, all corrections vanish and the isomorphism is the identity. Apply the construction to the transition matrices of $E_p$ on a neighbourhood of $p_0$ on which they are sufficiently close to those of $E$. This gives the stated family $\phi_p$. 
\end{proof}

%
%
By using this result, we obtain the following parametric version of 
Theorem \ref{th:TheoremA}. 

\begin{theorem}
\label{th:TheoremApar}
Assume that $\overline \Omega$ is a compact, strongly pseudoconvex domain with $\Cscr^2$ boundary in a Stein manifold $X$,  
$P$ is a compact Hausdorff space, 
and $E_p\to \overline \Omega$ $(p\in P)$ is a continuous
family of vector bundles of class 
$\lambda^r_\Oscr(\overline\Omega)$, $r>0$. 
There exist an integer $N\in \N$ and a 
continuous family of vector bundle epimorphisms
$\Phi_p: \overline \Omega\times \C^N \to E_p$, $p\in P$,
of class $\lambda^r_\Oscr(\overline \Omega)$.
\end{theorem}

\begin{proof}
By Theorem \ref{th:stability}, every point 
$p_0\in P$ has a neighbourhood $P_0\subset P$ and a continuous 
family of vector bundle isomorphisms $\phi_p:E_p\to E_{p_0}$ 
$(p\in P_0)$ of class $\lambda^r_\Oscr(\overline\Omega)$.
A vector bundle epimorphism
$\Psi_{p_0}: \overline \Omega\times \C^{n_0} \to E_{p_0}$,
given by Theorem \ref{th:TheoremA}, extends to a continuous
family of such epimorphisms $\Psi_p=\phi_p^{-1}\circ \Psi_{p_0}:
\overline \Omega\times \C^{n_0} \to E_{p}$ for $p\in P_0$.
Since $P$ is compact, we obtain finitely many pairs of open 
subsets $P'_i \Subset P_i\subset P$ and 
continuous families of vector bundle 
epimorphisms $\Psi^i_p:\overline \Omega\times \C^{n_i}\to E_p$ 
($p\in P_i$,  $i=1,\ldots,m$) such that $\bigcup_{i=1}^m P'_i=P$ .
For every $i=1,\ldots,m$ let $\chi_i:P\to [0,1]$ be a continuous
function such that $\chi_i=1$ on $P'_i$ and $\supp \chi_i\subset P_i$.
Let $\Phi^i_p: \overline \Omega\times \C^{n_i}\to E_p$ be 
defined by $\Phi^i_p(x,z)=\Psi^i_p(x,\chi_i(p)z)$ on $P_i$ and by zero elsewhere. 
Take $N=\sum_{i=1}^m n_i$. The family 
\[
	\Phi_p(x,z_1,\ldots,z_m)=\sum_{i=1}^m\Phi_p^i(x,z_i),\qquad\Phi_p: 
	\overline \Omega\times \C^{N}\to E_p,
	\quad p\in P,
\]
clearly satisfies the theorem. 
\end{proof}

%
%
We also have the following parametric version of 
Theorem \ref{th:complement}.

\begin{theorem}\label{th:complement-par}
Assume that $\Omega$ is as in Theorem \ref{th:complement}, 
$P$ is a compact Hausdorff space, and 
$\pi_p:E_p\to \overline\Omega$ is a continuous family 
of vector bundles of class $\lambda^r_\Oscr(\overline\Omega)$
(see Remark \ref{rem:stability}).  
Given a continuous family $E'_p\subset E_p$ $(p\in P)$ 
of  vector subbundles of class $\lambda^r_\Oscr(\overline\Omega)$, 
there is a continuous family $E''_p\subset E_p$ $(p\in P)$ 
of complementary vector subbundles of class 
$\lambda^r_\Oscr(\overline\Omega)$ such that 
\begin{equation}\label{eq:directsum} 
	E_p\cong E'_p\oplus E''_p,\quad p\in P.
\end{equation}
If in addition $Q$ is a closed subspace of $P$ which is a 
strong neighbourhood deformation retract and 
$E''_p\subset E_p$ $(p\in Q)$ is a continuous family 
of complex vector subbundles complementary to $E'_p$ 
of class $\lambda^r_\Oscr(\overline\Omega)$, then the family $E''_p$
can be extended to all values $p\in P$ so that 
\eqref{eq:directsum} holds.
\end{theorem}

\begin{proof}
Assume first that the family of vector bundles 
$E_p=E\to \overline \Omega$ is independent of $p\in P$.
Consider the short exact sequence of vector bundle homomorphisms 
\begin{equation}\label{eq:SES}	
	0\lra E'_p \lra E \stackrel{\phi_p}{\lra} E/E'_p \lra 0,
\end{equation}
where $\phi_p$ is the natural quotient projection for every $p\in P$. 
Theorem \ref{th:complement} gives for every $p\in P$ a 
subbundle $E''_p\subset E$ of class $\lambda^r_\Oscr(\overline\Omega)$
complementary to $E'_p$. Such a subbundle 
is the image of a unique vector bundle monomorphism 
$\sigma_p: E/E'_p\to E$ of the same class 
satisfying $\phi_p\circ \sigma_p=\Id$ on $E/E'_p$, and vice versa.
Such $\sigma_p$ is called a splitting of the sequence \eqref{eq:SES}.
Since $\overline \Omega$ is compact and the subbundles
$E'_p$ of $E$ vary continuously with $p$, 
a subbundle $E''_{p_0}\subset E$ complementary to $E'_{p_0}$
is also complementary to $E'_p$ for all $p\in P$  
near $p_0$, so it gives a continuous family of
splittings $\sigma_p:E/E'_p\to E$. Since a convex linear 
combination of splittings of $\phi_p$ in \eqref{eq:SES} 
is again a splitting, the proof follows by using a continuous
partition of unity on $P$.

The case of a continuously varying family $E_p\to \overline \Omega$, 
$p\in P$, is seen as follows. 
By Theorem \ref{th:stability}, every point $p_0\in P$ has a neighbourhood
$P_0\subset P$ and a continuous family of vector bundle isomorphisms 
$\phi_p:E_p\to E_{p_0}$ $(p\in P_0)$ of class 
$\lambda^r_\Oscr(\overline\Omega)$.
Hence, the family $E_p$ is locally trivial up to isomorphism, and 
our proof applies locally on $P$. It remains to patch 
the locally defined continuous families of 
splittings by a continuous partition of unity on $P$. 
For the relative assertion, view complements as right inverses
of the quotient maps. Local bundle identifications and a
neighbourhood retraction onto $Q$ extend the prescribed right
inverses near $Q$. Correct an extended map $\tau_p$ to
$\tau_p(\phi_p\tau_p)^{-1}$. The inverse exists near $Q$,
and the correction fixes the prescribed data. Patch these
local right inverses by a parameter partition of unity; affine
combinations of right inverses are right inverses. Finally use
a cutoff equal to $1$ on $Q$ to patch with a global family.
The images give the required complements.
\end{proof}

%
%
%
%
\section{A gluing lemma for sprays of Zygmund class 
$\lambda^r_\Oscr$}\label{sec:gluing}

Let $X$ be a complex manifold. Given a compact Lipschitz domain $K$ 
of $X$ and an open convex domain $0\in W\subset \C^N$ (the 
precise choice of $W$ is not important; for convenience
we may take a ball or a polydisc), we shall consider maps
$\gamma : K\times W\to K\times \C^N$ of the form
\begin{equation}\label{eq:gamma}
     \gamma(x,w) = \bigl(x,\psi(x,w)\bigr),\quad x\in K, \ w\in W
\end{equation}
and of class $\lambda^r_\Oscr(K\times W)$. 
Let $\Id(x,w)=(x,w)$ denote the identity map on $X\times \C^N$. 

The following splitting lemma is an analogue of 
\cite[Proposition 5.8.1, p.\ 235]{Forstneric2017E}
(see also \cite[Theorem 3.2]{DrinovecForstneric2008FM} 
and \cite[Lemma 3.2]{Forstneric2007AJM}), 
which pertains to function spaces $\Ascr^r(K)$ with $r\in \Z_+$. 

%
%
%
\begin{lemma}\label{lem:splitting}
Let $(A,B)$ be a Cartan pair in a Stein manifold $X$ 
(see Def.\ \ref{def:Cartan-pair}) and set $C=A\cap B$, $D=A\cup B$. 
Given a bounded convex domain $0\in W\subset\C^N$  
and a number $0<\epsilon<1$,  
there is a number $\delta>0$ satisfying the following.
For every map $\gamma : C\times W \to C\times \C^N$
of the form \eqref{eq:gamma} and of class 
$\lambda^r_\Oscr(C\times W)$ satisfying
$\|\gamma-\Id\|_{\lambda^r(C\times W)}<\delta$ 
there exist maps
\begin{equation}\label{eq:alphabeta}
 	 \alpha_\gamma : A \times \epsilon W \to A\times \C^N, \qquad
 	 \beta_\gamma :  B \times \epsilon W \to B \times \C^N
\end{equation}
of the form \eqref{eq:gamma} and of class
$\lambda^r_\Oscr(A\times \epsilon W)$ and 
$\lambda^r_\Oscr(B\times \epsilon W)$, respectively, 
depending smoothly on $\gamma$,
such that $\alpha_{\Id}=\Id$, $\beta_{\Id}=\Id$, and
\begin{equation}\label{eq:splitspray}
   \gamma\circ\alpha_\gamma = \beta_\gamma
   \quad \text{holds on $C\times \epsilon W$.}
\end{equation}
If $\gamma$ agrees with $\Id$ to order $k\in\N$ along $w=0$
then so do $\alpha_\gamma$ and~$\beta_\gamma$.
The analogous result holds with a continuous dependence
of maps on a parameter in a compact Hausdorff space.
\end{lemma}

\begin{proof}
We follow \cite[proof of Proposition 5.8.1]{Forstneric2017E}, 
taking into account \cite[Remark 5.8.3 (B), p.\ 238]{Forstneric2017E}
and Theorem \ref{th:dibar} in this paper. In particular, 
the use of \cite[Lemma 5.8.2, p.\ 236]{Forstneric2017E}
is replaced by Lemma \ref{lem:Cousin}. We recall the proof 
and adjust it to the little Zygmund classes.

Given a number $\epsilon\in (0,1]$, consider the Banach space
$C_{\epsilon}$ consisting of all continuous maps
\[
	\psi=(\psi_1,\ldots,\psi_N):C\times \epsilon W \to \C^N
\] 
which are holomorphic in $\mathring C\times \epsilon W$ and satisfy
\[
	\psi(\cdotp,w)\in \lambda^r(C,\C^N)\ \ 
	\text{for all $w\in \epsilon W$},
	\qquad
	\|\psi\|_{\epsilon} = \sup_{w\in \epsilon W} 
	\sum_{i=1}^N \|\psi_i(\cdotp,w)\|_{\lambda^r (C)} < \infty.
\]
Similarly, $A_{\epsilon}$ and $B_{\epsilon}$ denote Banach spaces
of the same kind associated to the sets $A$ and $B$, respectively. 
We may identify $\psi\in C_\epsilon$ with the bounded 
holomorphic map
\[
	\epsilon W\ni w \mapsto 
	\psi(\cdotp,w)\in  \lambda^r_\Oscr (C,\C^N).
\]
Note that $\lambda^r_\Oscr(C\times \epsilon W) \subset C_\epsilon$
and the inclusion is continuous. Furthermore, if $0<\epsilon'<\epsilon$ 
then the restriction map 
\begin{equation}\label{eq:restriction}
	C_\epsilon\ni \psi\mapsto \psi|_{C\times \epsilon'W} 
	\in \lambda^r_\Oscr(C\times \epsilon' W) 
\end{equation}
is continuous; see e.g. \cite[Theorem, p.\ 63]{BesovIlinNikolskii1979}. Let 
\[
	\Acal:\lambda^r_\Oscr(C) \to \lambda^r_\Oscr(A),
	\qquad  
	\Bcal:\lambda^r_\Oscr(C) \to \lambda^r_\Oscr(B)
\]
be the bounded linear operators in Lemma \ref{lem:Cousin}, satisfying
$\Acal c - \Bcal c=c$ for all $c\in \lambda^r_\Oscr(C)$ 
(see \eqref{eq:AplusB}). We extend $\Acal$ and $\Bcal$ 
to maps on $C_\epsilon$ defined by 
$(\Acal c)(\cdotp,w) = \Acal (c(\cdotp,w))$ for each $w\in\epsilon W$
and $c\in C_\epsilon$, and likewise for $\Bcal$. This gives bounded 
linear operators
\[
	\Acal:C_{\epsilon} \to A_{\epsilon},
	\qquad \Bcal:C_{\epsilon} \to B_{\epsilon},
\]
such that
\[
	\Acal c - \Bcal c=c \quad\text{for every $c\in C_\epsilon$.}
\]

We are given a map 
\begin{equation}\label{eq:gamma1}
	\gamma(x,w)=(x,\psi(x,w)),\quad x\in C,\ w\in W,
\end{equation}
of the form \eqref{eq:gamma}, with $\psi\in C_{1}$ close to the map
\[
	\psi_0(x,w)=w, 
\]
and a number $0<\epsilon <1$. We shall apply the construction 
by first increasing $\epsilon$ slightly in order to accommodate
a small decrease at the end of the proof. 
Given $c\in C_{\epsilon}$ near $0$, define 
\begin{equation}\label{eq:Phi-spray}
        \Phi(\psi,c)(x,w) =
        \psi\bigl(x,w+(\Acal c)(x,w)\bigr) - \bigl(w + (\Bcal c)(x,w)\bigr),
        \quad x\in C,\ w\in \epsilon W.
\end{equation}
Since $\Acal:C_\epsilon \to A_\epsilon$ is a bounded linear operator,
$\Phi(\psi,c)$ is well-defined for any $c\in C_{\epsilon}$ 
in a neighbourhood of $0$. 
We claim that $(\psi,c)\mapsto \Phi(\psi,c)$ is a smooth map from 
an open neighbourhood of $(\psi_0,0)$ in the Banach space 
$C_{1} \times C_{\epsilon}$ to $C_{\epsilon}$. 
To see this, choose $\epsilon_1$ with $\epsilon<\epsilon_1<1$. 
Banach-valued Cauchy estimates on $\epsilon_1 W$ bound all 
$w$-derivatives of $\psi$ on $\epsilon W$ in the mixed norm. 
For $\|c\|_{\epsilon}$ sufficiently small, the Taylor series 
\[ 
	\psi(x,w+(\Acal c)(x,w)) = \sum_{\nu\in\Z_+^N} 
	\frac{\partial_w^\nu\psi(x,w)}{\nu!} (\Acal c(x,w))^\nu 
\] 
converges in $C_{\epsilon}$. 
The same estimates justify termwise differentiation on a smaller
neighbourhood of $c=0$. Hence $\Phi$ is smooth as claimed. Note that 
\[
    \Phi(\psi_0,c)= \Acal(c)-\Bcal(c)=c,\quad c\in C_\epsilon,
\]
so $\di_c\, \Phi(\psi_0,0)$ is the identity map on $C_\epsilon$.
The implicit function theorem gives a smooth map 
\[
	\psi \longmapsto \Ccal(\psi) \in C_{\epsilon}, 
\]
defined in a neighbourhood of $\psi_0$ in $C_1$, such that 
\[
	\Phi(\psi,\Ccal(\psi))=0 \quad {\rm and} \quad 
	\Ccal(\psi_0)=0.
\]
The maps $a_\psi$ and $b_\psi$ defined by
\begin{equation}\label{eq:apsi}
        a_\psi(x,w) = w + \Acal \circ \Ccal(\psi)(x,w), \qquad
        b_\psi(x,w) = w + \Bcal\circ \Ccal(\psi)(x,w)
\end{equation}
then satisfy $a_\psi\in A_{\epsilon}$,  $a_{\psi_0}=\psi_0$, 
$b_\psi\in B_\epsilon$, $b_{\psi_0}=\psi_0$, and
\[
        \psi\bigl(x,a_\psi(x,w)\bigr) = b_\psi(x,w),
        \quad (x,w)\in C \times \epsilon W.
\]
Furthermore, after decreasing $\epsilon>0$ slightly, the comment 
after \eqref{eq:restriction} shows that 
\[
	a_\psi\in \lambda^r_\Oscr(A\times \epsilon W),
	\qquad 
	b_\psi\in \lambda^r_\Oscr(B\times \epsilon W).
\]
Since restriction is bounded linear, smooth dependence on $\gamma$ 
is preserved. The associated maps
\[
	\alpha_\gamma (x,w) = \bigl(x,a_\psi(x,w)\bigr),  \quad
	\beta_\gamma (x,w)  = \bigl(x,b_\psi(x,w)\bigr)
\]
depend smoothly on the map $\gamma$ \eqref{eq:gamma1} and satisfy condition \eqref{eq:splitspray}.

For jet preservation use the convention that agreement to
order $k$ means difference $O(|w|^{k+1})$ along $w=0$.
Let $C_\epsilon^{[k]}$ be the closed subspace of maps vanishing
to order $k$ in $w$. The operators $\Acal$ and $\Bcal$ act
only in $x$, so they preserve these subspaces. If
$\psi(x,w)-w=O(|w|^{k+1})$ and $c\in C_\epsilon^{[k]}$, then
\[
	 \Phi(\psi,c)=c+(\psi-\psi_0)
       \bigl(x,w+(\Acal c)(x,w)\bigr)\in C_\epsilon^{[k]}.
\]
The derivative in $c$ at $(\psi_0,0)$ is the identity on this
subspace. Apply the implicit function theorem there. Uniqueness
in the full space identifies the resulting solution with the
one already constructed. Hence $c$, $\Acal c$, and $\Bcal c$
vanish to order $k$. This works for 
continuous compact parameter families.
\end{proof}

%
%
Assume now that $D$ is a compact strongly 
pseudoconvex domain with $\Cscr^2$ boundary in a Stein manifold $X$,
$Y$ is a complex manifold, and $h:Z\to D$ is a topological fibre bundle
with fibre $Y$ which is holomorphic on $\mathring D=D\setminus bD$.
For $x\in D$ write $Z_x=h^{-1}(x)\cong Y$. The set
\[
	VT_z(Z) = T_z Z_{h(z)},\quad z\in Z 
\]
is called the vertical tangent space to $Z$ at $z$, 
and the vector bundle $VT(Z)\to Z$ with fibres $VT_z(Z)$ 
is the vertical tangent bundle of $(Z,h)$. 
When $h$ is differentiable in the variable $x\in D$, we have
$
	VT_z(Z) = \ker dh_z, \ z\in Z.
$
If the bundle $h:Z\to D$ is of Zygmund class $\lambda^r_\Oscr(D)$ 
(see Subsect.\ \ref{ss:Lr-bundles}) then $VT(Z)$ is of local 
class $\lambda^r_\Oscr(Z)$, 
and for every section $f:D\to Z$ of class $\lambda^r_\Oscr(D)$ the pullback
bundle $f^*VT(Z)\to D$ is a vector bundle of class $\lambda^r_\Oscr(D)$.
We recall the notion of (local) dominating sprays of sections;
see \cite[Definition 5.9.1, p.\ 239]{Forstneric2017E} 
for sprays over open domains and \cite[Sect.\ 8.10]{Forstneric2017E} 
for sprays over compact domains in a complex manifold. 

%
%
\begin{definition}\label{def:fibre-spray}
Let $h:Z\to D$ be a topological fibre bundle which is holomorphic on 
$\mathring D$. A holomorphic spray of sections of $h:Z\to D$  
is a continuous map $f : D\times W\to Z$ which is 
holomorphic on $\mathring D\times W$, where $0\in W\subset \C^N$ 
is an open convex set, such that 
\begin{equation} \label{eq:spray}
        h(f(x,w)) = x\quad \text{for $x \in D$ and $w\in W$}.
\end{equation}
The section $f_0=f(\cdotp,0):D\to Z$ is called the core of $f$.
The spray $f$ is dominating on a subset $K \subset D$ 
if the vertical derivative of $f$ at $w=0$, given by
\[
    	\di_w|_{w=0} f(x,w) :  T_0 \C^N\cong \C^N  \lra  VT_{f(x,0)} Z,
\]
is surjective for all $x\in K$. We say that $f$ is dominating if this 
holds for all $x\in D$. 
\end{definition}

We shall consider fibre bundles $h:Z\to D$ 
and sprays $f:D\times W\to Z$ of classes $\lambda^r_\Oscr(D)$ 
for $r>0$. The main ingredient in the proof of Theorem \ref{th:main}
is the following gluing lemma for such sprays. Its analogue
in spaces $\Ascr^r(D)$, $r\in\Z_+$, was first proved in 
\cite[Proposition 4.3]{DrinovecForstneric2007DMJ}. See also 
the more concise proofs in \cite[Proposition 2.4]{DrinovecForstneric2008FM} 
and \cite[Proposition 5.9.2]{Forstneric2017E}. 

%
%
%
%
\begin{lemma}[Gluing sprays of class $\lambda^r_\Oscr$]
\label{lem:gluing}
Let $(A,B)$ be a Cartan pair in a Stein manifold
(see Def.\ \ref{def:Cartan-pair}) and set $C=A\cap B$, $D=A\cup B$.
Assume that $h:Z\to D$ is a fibre bundle of class $\lambda^r_\Oscr(D)$.
Given a bounded ball or polydisc $0\in W_0\subset\C^N$ centred at $0$ and a spray of sections 
$f:A\times W_0\to Z$ of class $\lambda^r_\Oscr(A\times W_0)$ 
which is dominating on $C$, there is a ball $W\subset\C^N$ 
with $0\in W \subset W_0$ satisfying the following conditions. 
\begin{enumerate}[\rm (a)]
\item 
For every holomorphic spray of sections
$g : B \times W_0 \to Z$ of class $\lambda^r_\Oscr(B \times W_0)$
which is sufficiently close to $f$ in $\lambda^r(C \times W_0)$
there exists a spray of sections $f' : D \times W \to Z$
of class $\lambda^r_\Oscr(D\times W)$, close to $f$ in 
$\lambda^r(A\times W)$ (depending on the $\lambda^r$
distance between $f$ and $g$ on $C\times W_0$), 
whose core $f'_0$ is homotopic to $f_0=f(\cdotp,0)$ on $A$
and is homotopic to $g_0=g(\cdotp,0)$ on $B$.
\item
If $f$ and $g$ agree to order $k\in \N$ along $C\times\{0\}$, 
then $f'$ can be chosen to agree to order $k$ with $f$ 
along $A\times\{0\}$ and with $g$ along $B \times\{0\}$.
\end{enumerate}
The analogous result holds for families of sprays depending
continuously on a parameter $p\in P$ in a compact Hausdorff space.
If in addition the sprays $f$ and $g$ agree over $C$ for values of $p$ 
in a closed subset $Q\subset P$ which is a strong 
neighbourhood deformation retract, 
then $f'$ can be chosen to agree with $f$ and $g$ 
on $A$ and $B$, respectively, for $p\in Q$.
\end{lemma}

\begin{proof}[Sketch of proof]
The proof is analogous to that of 
\cite[Proposition 5.9.2, p.\ 240]{Forstneric2017E},
which applies to function spaces $\Ascr(D)$.
As noted in \cite[Remark 5.8.3 (B), p.\ 238]{Forstneric2017E},
the same proof applies in any Banach spaces on which there is a
linear bounded solution operator for the $\dibar$-equation 
on the level of $(0,1)$-forms. By Theorem \ref{th:dibar}, 
this holds for the little Zygmund spaces $\lambda^r$ with $r>0$. 
Nevertheless, some steps must be adjusted using results in  
Sections \ref{sec:approximation} and \ref{sec:vectorbundles}.
We indicate the necessary changes. 
The first step is to find a ball $0\in W\Subset W_0$ and a map 
$\gamma:C\times W\to C\times \C^N$ of the form 
\[
	\gamma(x,w) = (x,w + c(x,w)), \quad x\in C,\ w\in W,
\] 
with $c\in \lambda^r_\Oscr(C\times W)$ close to $0$
(depending on the $\lambda^r(C\times W_0)$ distance 
between $f$ and $g$), such that 
\begin{equation}\label{eq:fggama} 
	 f = g\circ \gamma \quad \text{holds on $C \times W$}.
\end{equation}
(For maps of class $\Ascr^r$ with $r\in \Z_+$, 
see \cite[Lemma 5.9.3, p.\ 240]{Forstneric2017E}.)  
To find $\gamma$, we split the trivial bundle $C\times \C^N$
by the subbundle $E'=\ker \di_w f|_{w=0}$,
the kernel of the vertical derivative of the spray $f$ at $w=0$. 
In our case, $E'$ is of class $\lambda^r_\Oscr(C)$, 
and a splitting is furnished by Theorem \ref{th:complement} in the 
basic case and by Theorem \ref{th:complement-par} in the parametric case. 
The proof of \cite[Lemma 5.9.3]{Forstneric2017E} 
then applies without further changes. 

In the second step, pick a number $0<\epsilon<1$
and apply Lemma \ref{lem:splitting} to split $\gamma$ in the form 
\[
   \gamma\circ\alpha = \beta
   \ \ \text{on $C\times \epsilon W$},
\]
with $\alpha$ and $\beta$ as in \eqref{eq:alphabeta}. 
From this and \eqref{eq:fggama} it follows that 
\begin{equation}\label{eq:falphagbeta}
	f\circ \alpha = g\circ \beta \ \ \text{on $C\times \epsilon W$}.
\end{equation}
Hence, the two sides of the above equation amalgamate 
to a spray $f':D\times \epsilon W\to Z|_D$ satisfying the lemma. 
The proof in the parametric case follows the same scheme.

\end{proof}

%
%
%
%
\section{Proof of Theorem \ref{th:main}}\label{sec:proofmain}

We begin by explaining the proof of the basic
(nonparametric) case of Theorem \ref{th:main}.
The parametric version in Theorem \ref{th:mainpar}
includes the 1-parametric case stated in Theorem \ref{th:main}. 

Let $h:Z\to \overline \Omega$ be a fibre bundle of class 
$\lambda^r_\Oscr(\overline \Omega)$ with an Oka fibre $Y$.
Our task is to prove that any continuous section 
$f_0\in \Gamma(\overline \Omega,Z)$ is homotopic to a section 
$f\in \Gamma^r_\Oscr(\overline \Omega,Z)$ which is holomorphic
on $\Omega$. Choose a $\Cscr^2$ strongly plurisubharmonic
function $\rho$ on a neighbourhood $U\subset X$ of $\overline\Omega$
such that 
\[
	\Omega=\{x\in U:\rho(x)<0\},\quad  
	d\rho_x\ne 0 \ \ \text{for every $x\in b\Omega =\{\rho=0\}$.}
\]
Pick $c<0$ such that $\rho$ has no critical values on the interval $[c,0]$
and set $A_0=\{\rho\le c\}$. Given an open cover $\Ucal=\{U_j\}$ of 
$\overline{\overline\Omega\setminus A_0}$ 
consisting of holomorphic coordinate charts $U_j\subset X$,  
\cite[Lemma 5.10.3]{Forstneric2017E} gives compact, strongly pseudoconvex domains with $\Cscr^2$ boundary 
\begin{equation}\label{eq:sequenceAbis}
	A_0 \subset A_1 \subset \cdots \subset A_m=\overline \Omega
\end{equation}
for some $m\in\N$ such that for every $k=0,1,\ldots,m-1$ we have 
$A_{k+1}=A_k\cup B_k$, where $(A_k,B_k)$ is a special Cartan pair
(see Definition \ref{def:Cartan-pair}) and $B_k\subset U_j$ for some $j=j(k)$.  
Hence, we may assume that the bundle $Z\to \overline \Omega$
is trivial over $B_k$ for all $k=0,1,\ldots,m-1$.

By the Oka principle on open Stein manifolds 
\cite[Theorem 5.4.4]{Forstneric2017E}, we may assume that 
$f_0$ is holomorphic on a neighbourhood of $A_0$. 
It suffices to show that for every special 
Cartan pair $(A,B)$ in $\overline \Omega$ such that the bundle
$Z\to \overline \Omega$ is trivial over $B$, 
we can approximate a section $f\in \Gamma^r_{\Oscr}(A,Z)$ as closely 
as desired by sections $\tilde f\in \Gamma^r_{\Oscr}(D,Z)$, 
where $D=A\cup B$. The theorem then follows by a finite induction,
using the sequence \eqref{eq:sequenceAbis} and starting with 
the section $f_0$ on $A_0$. The existence of a homotopy from $f_0$ 
to $f=f_m$ is obvious since there is no change of topology from $A_0$
to $\overline \Omega$.

Fix a special Cartan pair $(A,B)$ and a section $f\in \Gamma^r_{\Oscr}(A,Z)$. 
The construction of a section $\tilde f \in \Gamma^r_{\Oscr}(D,Z)$ 
approximating $f$ in $\Gamma^r_{\Oscr}(A,Z)$ proceeds in three steps.
\begin{enumerate}[\rm (1)]
\item We embed $f$ as the core of a dominating spray of sections
$F:A\times W\to Z$ of class $\lambda^r_{\Oscr}(A\times W,Z)$
(see Definition \ref{def:fibre-spray}), where $0\in W\subset \C^N$ 
is a ball. The construction of $F$ is explained in the sequel. 
\item 
Set $C=A\cap B$ and fix a number $\delta\in (0,1)$. 
In the trivialisation over $B$, view $F$ as a map to the
Oka fibre $Y$. Dilate only the base variable in convex
coordinates on a neighbourhood of $C$, keeping $w$ in a slightly 
larger parameter ball. Lemma \ref{lem:interpolation-dilation} and 
Cauchy estimates give convergence in the mixed little norm, hence in the joint
little norm on a smaller parameter ball. The dilated map is
holomorphic on a neighbourhood of $C$ times that ball. Apply
CAP on a slightly thickened convex compact product and then
Cauchy estimates. This gives a map holomorphic near
$B\times\delta W$ and arbitrarily close to $F$ in
$\lambda^r(C\times\delta W)$. It defines the required spray
$G$ in the bundle trivialisation. 

\item 
Assuming that the approximation in the previous step 
is close enough, we apply Lemma \ref{lem:gluing} to glue $F$ 
and $G$ into a spray $\wt F\in \Gamma^r_\Oscr(D\times \delta' W)$ 
for some $\delta'\in (0,\delta)$ which approximates $F$ 
in $\lambda^r(A\times \delta' W)$. The section
$\tilde f:= \wt F(\cdotp,0) \in \Gamma^r_{\Oscr}(D,Z)$ 
then approximates $f$ in $\Gamma^r_{\Oscr}(A,Z)$. 
\end{enumerate}

This completes the proof, modulo the construction of a 
dominating spray in step (1). For this, we follow 
\cite[proof of Proposition 8.10.2, p.\ 388]{Forstneric2017E}
(the original reference is \cite[Proposition 4.1]{DrinovecForstneric2008FM}), 
where a result of this kind was proved for classes $\Ascr^r$ with $r\in\Z_+$.
Here is a sketch.

By Remark \ref{rem:TheoremA}, there is a vector bundle epimorphism 
$
	L:A\times \C^N\to f^* VT(Z)
$ 
of class $\lambda^r_\Oscr(A\times\C^N)$ for some $N\in \N$.

\noindent {\bf Claim:} Given $L$ as above, there is a dominating spray 
$F:A\times W\to Z$ of class $\lambda^r_{\Oscr}(A\times W,Z)$,
where $0\in W\subset \C^N$ is a ball, such that 
$F(\cdotp,0)=f$ and 
\begin{equation}\label{eq:L}
    	\di_w|_{w=0} F(x,w) = L(x,\cdotp): \C^N  \to VT_{f(x)}Z
	\ \ \text{for all $x\in A$}.
\end{equation}

To prove the claim, take a finite convex-bump decomposition
of $A$ as in \eqref{eq:sequenceAbis} with $A_m=A$, 
starting at a compact strongly pseudoconvex domain 
$A_0\subset \mathring A$. On $A_0$ the usual Stein spray construction 
gives a holomorphic spray satisfying \eqref{eq:L}.
Choose the bumps small enough that in a bundle trivialisation over
each bump $B$ the fibre component of $f(B)$ lies in a holomorphic 
fibre coordinate chart $\theta$. On that bump set
\[
 	H(x,w)=\theta^{-1}\bigl(\theta(f(x))+d\theta_{f(x)}L(x,w)\bigr),
	\quad x\in B,
\]
interpreted in that trivialisation. On a sufficiently small
parameter ball this is a spray of class $\lambda^r_\Oscr$ 
with core $f$ and vertical derivative $L$ as in \eqref{eq:L}. 

Suppose inductively that a spray $F_k$ satisfying \eqref{eq:L}
has been constructed on $A_k$.
On $C_k=A_k\cap B_k$, $F_k$ and the local spray $H_k$ have the
same core and first derivative, so their difference in fibre
coordinates is $O(|w|^2)$, with $\lambda^r_\Oscr$ coefficients in $x$.
The transition-map construction of Lemma \ref{lem:gluing},
using a complement to $\ker L$, gives a transition tangent to
the identity with error $O(|w|^2)$. On writing $w=\delta\zeta$,
its normalised error is $O(\delta)$ in the mixed little norm
on a fixed $\zeta$-ball. It is therefore small enough
for Lemma \ref{lem:splitting}. Its jet-preserving part gives
reparametrisations equal to the identity to order $1$.
The reparametrised sprays glue on $A_k\cup B_k$, preserving
the core and first derivative. The induction is completed in 
finitely many steps. For compact continuous
parameter families use Theorems \ref{th:TheoremApar} and
\ref{th:complement-par}, a finite cover in base and parameter,
and uniform choices of the finitely many radii.

This completes the proof of the basic case of Theorem \ref{th:main}.

%
%
We now state a general parametric version of Theorem \ref{th:main}.
Recall that $\Gamma(\overline \Omega,Z)$ denotes the space of
continuous sections of a topological fibre bundle $h:Z\to \overline \Omega$. 
If the bundle is of class $\lambda^r_\Oscr(\overline \Omega)$ then 
$\Gamma^r_\Oscr(\overline \Omega,Z)$ denotes the space of
sections of the same class.

%
%
\begin{theorem}[Relative parametric little-space Oka principle]
\label{th:mainpar}
Assume that $\Omega$ is as in Theorem \ref{th:main},
$r>0$, $h:Z\to \overline \Omega$ is a fibre bundle of class 
$\lambda^r_\Oscr(\overline \Omega)$ with Oka fibre $Y$, 
$P$ is a compact Hausdorff space, and $Q$ is a closed subset 
of $P$ which is a strong neighbourhood deformation retract.
Let $f:P\to \Gamma(\overline \Omega,Z)$ be a continuous map
such that $f|_Q:Q\to \Gamma^r_\Oscr(\overline \Omega,Z)$
is continuous in the topology induced by the local $\Lambda^r$ norms.
Then there is a homotopy 
$f_s:P\to \Gamma(\overline \Omega,Z)$, $s\in [0,1]$,
which is fixed on $Q$ such that $f_0=f$ and 
$f_1:P\to  \Gamma^r_\Oscr(\overline \Omega,Z)$ is continuous
in the topology induced by the local $\Lambda^r$ norms.
\end{theorem}

%
%
%
%
\begin{proof} 
All spaces of little Zygmund maps 
and sections below carry the topology induced by the local $\Lambda^r$ 
norms. For sprays we work, in local coordinates, with the mixed norm 
\[ \|F\|_{r,W}=\sup_{w\in W}\|F(\cdot,w)\|_{\Lambda^r} \] 
on a spray parameter ball $W$. We use the passage from mixed to joint  
regularity on smaller balls, as in \eqref{eq:restriction}. In particular, continuity of 
the spray families in the mixed norm gives continuity in the joint 
little-Zygmund-space topology after restriction to a smaller ball. 
All shrinkings below are finite in number. 

\medskip\noindent \emph{1. Section charts and families of sprays.} 
Let $A\subset \overline\Omega$ be a compact strongly pseudoconvex 
domain with $\Cscr^2$ boundary, and let $u\in\Gamma^r_\Oscr(A,Z)$. 
The prescribed-derivative spray construction in the proof of 
Theorem~\ref{th:main} gives a dominating spray with core $u$. 
By Theorem~\ref{th:complement}, the kernel of its vertical derivative along 
the core has a little-class holomorphic complement. Restricting the spray 
to this complement and identifying the complement with $E=u^*VT(Z)$ 
by the vertical derivative gives a neighbourhood $\mathcal V\subset E$ 
of the zero section and a fibre-preserving map 
\[ 
	S:\mathcal V\longrightarrow Z|_A,  \quad 
	S(x,0)=u(x),\quad d_vS(x,0)=\Id.
\] 
The map is of little class in the base variable and holomorphic in the fibre 
variable. After shrinking $\mathcal V$ to a fibrewise convex neighbourhood 
of the zero section, $S$ is a fibrewise biholomorphism onto a neighbourhood 
of $u(A)$. In local bundle coordinates the fibrewise inverse has the same
regularity, by the implicit function theorem in the mixed little norms, 
applied on smaller fibre balls. Thus $S$ and its inverse induce continuous 
chart operations on little-class holomorphic sections and on sprays, and 
also on continuous sections in the uniform topology. In such a chart, two 
nearby sections $v_0,v_1$, written as $v_i(x)=S(x,\xi_i(x))$, are joined by 
a homotopy of sections
\[ 
	x\longmapsto S\bigl(x,(1-\tau)\xi_0(x)+\tau\xi_1(x)\bigr), 
	\qquad 0\le\tau\le1. 
\] 
This interpolation is constant if $v_0=v_1$, and it commutes with restriction 
to a smaller base domain. The same formula applies to sprays, with $\xi_i$ 
depending holomorphically on the spray variable. In that case it preserves 
the common core and, when the first vertical derivatives agree, their common 
first vertical derivative. We give the details of the relative spray construction 
that will be used in the induction. Suppose that 
$p\mapsto u_p\in\Gamma^r_\Oscr(A,Z)$ is continuous on $P$. 
Assume that, on $Q$, a continuous family of dominating sprays $F_q^*$ 
with cores $u_q$ is given on a common parameter ball in $\C^{N_*}$. 
We claim that there is a continuous family of dominating sprays on $A$, 
with cores $u_p$, which agrees on $Q$ with the prescribed family after 
adjoining inactive variables. First extend the prescribed sprays to a 
parameter neighbourhood of $Q$, with cores $u_p$. Let $\sigma$ be a 
neighbourhood retraction onto $Q$. In a section chart $S$ centred at 
$u_{q_0}$, write $u_p(x)=S(x,\xi_p(x))$ for $p$ near $q_0$. For such $p$, 
define a local extension by 
\[ 
	H_p(x,w)
	= S\bigl(x,\xi_p(x)+\eta_{\sigma(p)}(x,w) -\xi_{\sigma(p)}(x)\bigr),
 	\qquad F^*_{\sigma(p)}(x,w)=S(x,\eta_{\sigma(p)}(x,w)). 
\] 
On a smaller common spray ball this is well defined, has core $u_p$, and 
equals $F_p^*$ for $p\in Q$. Interpolate these local extensions in section 
charts using continuous parameter cutoffs. For clarity, the finite patching 
can be done by choosing compact parameter subsets inside the chart 
neighbourhoods whose interiors cover $Q$. At each stage patch the family 
already defined near the preceding compact union with the new local 
family near the next compact subset. Use a partition of unity on the 
union of the two parameter neighbourhoods, with interpolation only on 
their overlap. Shrink the neighbourhoods around the compact subsets and 
the common spray ball as necessary. Both families have core $u_p$, so on 
the overlap both lie in the new section chart on a sufficiently small 
uniform ball. The interpolation does not change the prescribed sprays 
on $Q$. The resulting family, denoted by $H_p$, is dominating on a 
neighbourhood of $Q$ after shrinking that neighbourhood, by 
compactness and openness of surjectivity of the vertical derivative. 
Write $L_p^*=\partial_w H_p|_{w=0}$ on this neighbourhood. By 
Theorem~\ref{th:TheoremApar} there is a continuous family of 
epimorphisms 
\[ 
	M_p:A\times\C^{N_0}\longrightarrow u_p^*VT(Z),\qquad p\in P. 
\] 
Choose a continuous cutoff $\chi:P\to[0,1]$ equal to $1$ near $Q$ and 
supported in the neighbourhood where $H_p$ is defined and dominating. 
Set 
\[ 
	L_p(x,w,z)=\chi(p)L_p^*(x,w)+(1-\chi(p))M_p(x,z), 
	\qquad (w,z)\in\C^{N_*}\times\C^{N_0}, 
\] 
extending the first term by zero off that neighbourhood. The variables in the 
two summands are independent, so $L_p$ is surjective for every $p$. 
Near $Q$ it is the first vertical derivative of $H_p(x,w)$ viewed as a spray in 
$(w,z)$. For any $p_0\in P$, a section chart $S$ centred at $u_{p_0}$ gives, 
for $p$ close to $p_0$, a spray with core $u_p$ and first vertical derivative 
$L_p$ by the formula 
\[ 
	(x,v)\longmapsto S\bigl(x,\xi_p(x)
	+(d_vS(x,\xi_p(x)))^{-1}L_p(x,v)\bigr), \qquad v=(w,z). 
\] 
Patch these local families as above, starting with $H_p$ on a smaller 
neighbourhood of $Q$ where $\chi=1$. The subsequent cutoffs can be 
chosen to vanish on $Q$: the remaining compact parameter set to be 
covered is disjoint from a neighbourhood of $Q$. Every interpolation preserves 
the common first vertical derivative $L_p$, hence domination. This proves the 
relative claim. When $Q=\varnothing$, simply use $M_p$ and the local formula 
and patching just described. We shall also use the following consequence of 
section charts. A continuous family of little-class holomorphic sections 
given on $Q$ may be extended as such to a parameter neighbourhood by 
composing with $\sigma$. If another family is sufficiently close there, we can 
change it to the prescribed family on $Q$ by the above interpolation with 
parameter cutoffs. With finitely many charts centred at members of the 
prescribed family, take the cutoffs equal to $1$ on smaller sets covering $Q$. 
Once equality is obtained it is preserved by subsequent interpolations. 
The same construction works for sprays on a smaller ball. If the two families 
are also close after restriction to a fixed base subdomain, the correction on 
that subdomain is small in its own norm: the chart operations restrict to it 
and are continuous in that norm. In particular, by taking the charts on a 
fixed neighbourhood of the compact prescribed family, these restricted 
estimates do not depend on the behaviour of the other family away from 
the support of the parameter cutoffs. 

\medskip\noindent \emph{2. Relative approximation on a convex bump.} 
Let $C\subset B$ be the convex compact sets in the coordinates of a special 
Cartan pair, and suppose the bundle is trivial over $B$. In the trivialisation, 
let $a_p$ be a continuous family of little-class holomorphic maps on $C$, 
$p\in P$, and let $b_q$ be such a family on $B$, $q\in Q$, with $b_q|_C=a_q$. 
We show that $a_p$ can be approximated uniformly in $p$ in the little norm on 
$C$ by a continuous family on $B$ of the same class, equal to $b_q$ for 
$q\in Q$. The argument also applies to sprays, with the mixed norms on a 
smaller common spray ball. Choose the origin in $\mathring C$ and use this 
same centre for both dilations. Set 
\[
 	a_{p,t}(x)=a_p(tx),\qquad b_{q,t}(x)=b_q(tx), \qquad 0<t<1. 
\] 
For a fixed $t<1$, these families are holomorphic on common neighbourhoods 
of $C$ and $B$, respectively. They agree for $q\in Q$ on a common 
neighbourhood of $C$, since $tx\in \mathring C$ for $x$ in a sufficiently small 
neighbourhood of $C$. By Lemma~\ref{lem:interpolation-dilation}, $a_{p,t}$ 
converges to $a_p$ in the little norm on $C$ as $t\uparrow1$, 
uniformly in $p$, and $b_{q,t}$ converges to $b_q$ in that norm on $B$, 
uniformly in $q$. For maps to $Y$, apply the scalar estimates in finitely many 
local coordinates on compact pieces of the graphs. 
For sprays, dilate only the base variable and work with the spray 
variable in a slightly smaller ball. The closed smaller ball and the compact 
external parameter spaces give compact families of the coefficient maps to 
which the dilation lemma applies. Fix $t<1$ close enough to $1$. Choose a 
convex open neighbourhood $U$ of $B$ on which the second family is defined 
and compact convex neighbourhoods $L_0,L_1$ of $C$ with 
\[ 
	C\subset\mathring L_0\subset L_0
	\subset\mathring L_1 \subset L_1\subset U, 
\] 
such that the first family is defined near $L_1$ and the two families agree 
there for $q\in Q$. For sprays, use the corresponding products with a common 
open spray ball in the source and nested closed smaller balls in the 
approximation sets. These are convex products; no boundary regularity 
theorem on the products is needed. In the remainder of this paragraph, $U,L_0,L_1$ denote the resulting convex sets in either case. 
Let $\pi:U\to L_1$ be nearest-point projection. The maps $a_{p,t}\circ\pi$ 
give a continuous family on $U$. On $Q\times U$, the homotopy 
\[ 
	(q,x,s)\longmapsto b_{q,t}\bigl((1-s)\pi(x)+sx\bigr), 
	\qquad 0\le s\le1, 
\] 
starts at this family and is stationary on $Q\times L_1$. The homotopy 
extension property for the cofibration 
\[ 
	(Q\times U)\cup(P\times L_1)\ \hookrightarrow\ P\times U 
\] 
extends the homotopy to $P\times U$, keeping the family on $P\times L_1$ 
fixed. Here we use that $(P,Q)$ is a neighbourhood deformation retract pair 
and that the inclusion of a compact convex subset in an open convex set 
is a cofibration. We thus obtain a continuous family on $U$ which agrees 
with $a_{p,t}$ on $L_1$ and with $b_{q,t}$ on $U$ for $q\in Q$. 
Apply the relative parametric 
Oka principle for the Oka manifold $Y$ on the convex Stein domain $U$, 
with approximation on $L_0$ and with the family fixed on $Q$ 
\cite[Theorem 5.4.4]{Forstneric2017E}. This gives a continuous family 
$g_p\in\Oscr(U,Y)$ with $g_q=b_{q,t}$, as close to $a_{p,t}$ on $L_0$ as 
desired. Cauchy estimates in local target coordinates give approximation in 
the $\Lambda^r$ norm on $C$, and in the spray case in the mixed norm 
on $C$ times a smaller spray ball. In particular, $g_p|_C$ is close to $a_p$ 
in the required norm and $g_q|_B$ is close to $b_q$ in that norm on $B$. 
Finally extend the prescribed $b_q$ to a parameter neighbourhood by the 
retraction $\sigma$ and use the chart correction in part~1, with the charts 
centred at members of the prescribed family on $B$. The correction is defined 
since $g_q|_B$ is close to $b_q$ on $B$, and this closeness persists on a 
sufficiently small parameter neighbourhood of $Q$. On $C$, both $g_p|_C$ and 
$b_{\sigma(p)}|_C=a_{\sigma(p)}$ are close to $a_p$ on this neighbourhood, the 
latter by continuity after shrinking the neighbourhood. The restricted chart 
estimates therefore keep the error on $C$ arbitrarily small. The corrected family 
is of little class on $B$, is continuous in the required topology, and equals $b_q$ for $q\in Q$. This proves the relative approximation assertion. 

\medskip\noindent \emph{3. Preparation on the interior.} 
Choose a $\Cscr^2$ strongly plurisubharmonic defining function $\rho$ on a 
neighbourhood of $\overline\Omega$, and a number $c<0$ so close to 
$0$ that $\rho$ has no critical values in $[c,0]$. Set $A_0=\{\rho\le c\}$. 
Choose a finite convex-bump decomposition as in the proof of Theorem~\ref{th:main}, 
\[ 	A_0\subset A_1\subset\cdots\subset A_m
	=\overline\Omega,\qquad A_{j+1}=A_j\cup B_j, 
\] 
with $Z$ trivial over each $B_j$ in its little-class holomorphic trivialisation. 
Apply the relative parametric Oka principle on $\Omega$ to $f|_\Omega$, 
keeping the family fixed on $Q$. For the interior approximation assertion in the 
introduction, take $A_0$ to contain the given compact set in its interior and use 
Oka approximation there in the preliminary homotopy. The subsequent bump 
homotopies are holomorphic on $\mathring A_0$ and can be made arbitrarily 
small on the given compact set. The terminal family is holomorphic on 
$\Omega$, and its restrictions to $A_0\Subset\Omega$ form a continuous 
family in $\Gamma^r_\Oscr(A_0,Z)$ by Cauchy estimates. Restrict the 
homotopy to $A_0$ and extend it to a homotopy of section families on 
$\overline\Omega$, relative to $Q$. Here and below we use the section form 
of the homotopy extension property: for $A'\subset \overline\Omega$ a 
closed cofibration, the restriction of continuous sections from 
$\overline\Omega$ to $A'$ has the homotopy lifting property, also relative 
to the cofibration $Q\hookrightarrow P$. This follows from the relative covering 
homotopy property of the fibre bundle, applied to the inclusion 
\[ 
	(P\times A')\cup(Q\times \overline\Omega)\ 
	\hookrightarrow\ P\times \overline\Omega. 
\] 
These inclusions are closed cofibrations, even when the boundaries
touch. For nested $\Cscr^2$ domains the inward normals agree at
common boundary points. Using local inward directions and a smooth
partition of unity, choose a vector field near the smaller boundary
strictly inward for both domains along the common boundary and
strictly inward for the smaller domain elsewhere. By shrinking the
neighbourhood, it points inward along the part of the larger boundary
in that neighbourhood. Its flow up to the first hitting time of the
smaller boundary retracts an outer collar in the larger domain.
The hitting time is continuous by transversality and tends to zero
at the smaller boundary. Fix the smaller domain and cut off in a
slightly larger collar to obtain a neighbourhood deformation retract
pair. Let the endpoint of this first homotopy be the current family on 
$\overline\Omega$. 

\medskip\noindent \emph{4. Induction across a bump.} 
Suppose that the current family on $\overline\Omega$ is fixed at $f(q)$ for 
$q\in Q$, and that its restriction to $A=A_j$ is a continuous family $u_p\in\Gamma^r_\Oscr(A,Z)$. Put 
\[ 
	B=B_j,\qquad C=A\cap B,\qquad D=A\cup B. 
\] 
By part~1 choose a continuous family of dominating sprays on $D$ with cores 
$f(q)|_D$, $q\in Q$. Apply the relative spray construction of part~1 on $A$ to 
obtain dominating sprays $F_p$ with cores $u_p$, agreeing on $Q$ with the 
restrictions of the prescribed sprays. Adjoin the same inactive variables to the 
prescribed sprays on $D$. In the trivialisation over $B$, apply part~2 to 
$F_p|_C$, using the prescribed sprays on $B$ for $q\in Q$. On a smaller 
common parameter ball this gives sprays $G_p$ over $B$, as close as desired 
to $F_p$ on $C$ in the little norm and equal to the prescribed sprays on $B$ for 
$p=q\in Q$. Apply Lemma~\ref{lem:gluing}. The transition map can be chosen 
to be the identity on $Q$: in its implicit-function construction the zero correction 
solves the transition equation when the input sprays agree, and the small 
solution in the chosen complement is unique. The splitting operators in 
Lemma~\ref{lem:splitting} then give identity factors on $Q$. Hence the glued 
cores form a continuous family 
\[ 
	v_p\in\Gamma^r_\Oscr(D,Z),\qquad v_q=f(q)|_D\quad(q\in Q), 
\] 
with $v_p|_A$ as close to $u_p$ as desired. We verify that $v_p$ can be 
obtained from the current family by a homotopy fixed on $Q$. 
First join $u_p$ to $v_p|_A$ by a short fibrewise homotopy on $A$, 
stationary for $q\in Q$. Use the charts from part~1 and parameter cutoffs: 
choose finitely many parameter neighbourhoods on which the two families 
belong to a common chart and a partition of unity $(\chi_i)$ subordinate to them. 
Starting at $u_p$, move towards $v_p|_A$ in the $i$th chart with interpolation 
length $\chi_i(p)/(1-\sum_{j<i}\chi_j(p))$ where the denominator is positive. 
Where an earlier step has reached the endpoint, stay there. This is continuous at 
parameters where the denominator vanishes, since the preceding map is already 
the endpoint there and every short interpolation fixes the diagonal. Uniform 
closeness keeps all stages in the relevant charts. In particular the homotopy is 
holomorphic on $\mathring A$ and can be made uniformly small there. In the 
trivialisation over $B$, prescribe the old and new maps on $B\times\{0,1\}$ and 
the homotopy just constructed on $C\times[0,1]$. Write 
\[ 
	E_B=(B\times\{0,1\})\cup(C\times[0,1]). 
\] 
The inclusion $E_B\hookrightarrow B\times[0,1]$ is a cofibration and a 
homotopy equivalence. Indeed, $C$ is a strong deformation retract of $B$ by 
nearest-point projection and straight-line interpolation, and the same 
deformation, leaving the time coordinate unchanged, deforms both spaces 
onto $C\times[0,1]$. Taking the product pair with $(P,Q)$, the inclusion 
\[ 
	(P\times E_B)\cup(Q\times B\times[0,1]) 
	\ \hookrightarrow\ P\times B\times[0,1] 
\] 
is again a cofibration and a homotopy equivalence, by the product-pair 
property of cofibrations. Thus $P\times B\times[0,1]$ retracts onto the 
subspace displayed above. The prescribed maps on this subspace, 
including the stationary homotopy on $Q\times B\times[0,1]$, 
agree on all overlaps. Composing with a retraction extends them to the 
required homotopy on $B$. Gluing with the homotopy on $A$ gives a 
homotopy on $D$ with the prescribed initial and terminal families. Extend 
it from $D$ to $\overline\Omega$ by the section homotopy extension property
used in part~3, relative to $Q$. The new current family on $\overline\Omega$ 
restricts on $D=A_{j+1}$ to $v_p$. This completes the induction step. 
After the final step the current family is a continuous map 
$P\to\Gamma^r_\Oscr(\overline\Omega,Z)$ in the inherited $\Lambda^r$ 
topology and equals $f$ on $Q$. Concatenating the preliminary homotopy 
and the finitely many bump homotopies gives the required homotopy from 
$f$ to this family, fixed on $Q$. 
\end{proof}

%
%
\begin{theorem}[Relative parametric big-space Oka principle]
\label{th:mainpar-big}
Let $\Omega$ and $h:Z\to\overline\Omega$ satisfy the hypotheses
of Theorem \ref{th:main-big}. Let $P$ be compact Hausdorff and
$Q\subset P$ a closed strong neighbourhood deformation retract.
Suppose $f:P\to\Gamma(\overline\Omega,Z)$ is continuous and
$f|_Q$ consists of $\Lambda^r_\Oscr$ sections, continuously in
the $\Lambda^r$ topology. For every prescribed $0<s<r$ there
is a homotopy of continuous section families, fixed on $Q$,
from $f$ to a family of $\lambda^{s}_\Oscr$ sections continuous
in the inherited $\Lambda^{s}$ topology.
\end{theorem}

\begin{proof}
Apply the continuous inclusion
$\Lambda^r\hookrightarrow\lambda^{s}$ to the bundle transition
maps and to the prescribed family on $Q$. Then apply
Theorem \ref{th:mainpar} at exponent $s$. The inclusion is
continuous, so the required parameter continuity on $Q$ is preserved. 
\end{proof}

%
%
%
%
\section{The Oka principle for vector bundles and principal
bundles of class $\lambda^r_\Oscr$}
\label{sec:OPVB}

In this section, we prove Theorem \ref{th:OPVB}.
We follow \cite[proof of Theorem 5.3.1]{Forstneric2017E},
which is due to Grauert \cite{Grauert1958MA}. 
See also Cartan's exposition of Grauert's Oka principle in \cite{Cartan1958}.

\begin{proof}[Proof of Theorem \ref{th:OPVB}]
Denote by $G(m,N)$ the Grassmann manifold 
of all complex $m$-dimensional subspaces of $\C^N$.
A topological vector bundle $E \to \overline\Omega$ of rank $m$ is 
the pullback $f^*\U$ of the universal bundle $\U\to G(m,N)$ of rank $m$
by a continuous map $f:\overline\Omega\to G(m,N)$ for a suitable $N$.
(We take $N$ big enough such that $E$ embeds as a topological
vector subbundle of the trivial bundle $\overline\Omega \times \C^N$.)
Since $G(m,N)$ is a complex homogeneous manifold, and hence an 
Oka manifold by Grauert \cite{Grauert1957II}, 
Theorem \ref{th:main} shows that $f$ is 
homotopic to a map $F:\overline\Omega \to G(m,N)$ 
of class $\lambda^r_\Oscr(\overline\Omega)$.
The pullback $F^*\U\to \overline\Omega$ is then a 
vector bundle of class $\lambda^r_\Oscr(\overline\Omega)$
which is topologically isomorphic to $E\cong f^*\U$. 
This proves part (i) of the theorem.

To prove the second statement, let $E\to \overline\Omega$ and 
$E' \to \overline\Omega$ be vector bundles of class
$\lambda^r_\Oscr(\overline\Omega)$ and rank $m$. 
The bundle $Z=\operatorname{Isom}_{\C}(E,E')$ of fibrewise linear
isomorphisms has fibre $GL_m(\C)$. In local frames its transitions
are $v\mapsto g'_{ij}v g_{ji}$, where $g_{ij}$ and $g'_{ij}$ are
the transitions of $E$ and $E'$. The bundle is therefore of class
$\lambda^r_\Oscr$. A topological isomorphism is a continuous section
of $Z$, and a little-class holomorphic section is an isomorphism
of the required class. Since $GL_m(\C)$ is Oka, Theorem~\ref{th:main}
applies.
Hence, part (ii) follows from Theorem \ref{th:main}.
\end{proof}

Similarly one can prove the following analogue of 
\cite[Theorem 8.2.1]{Forstneric2017E} due to Grauert \cite{Grauert1958MA}.
One uses the analogue of Lemma \ref{lem:Cartan} for an
arbitrary complex Lie group $G$; see Remark \ref{rem:Liegroup}.

%
%
\begin{theorem}\label{th:OPprincipalbundles}
Let $\Omega$ be as in Theorem \ref{th:OPVB}. 
For every complex Lie group $G$, the isomorphism classes of 
principal $G$-bundles on $\overline\Omega$ 
of class $\lambda^r_\Oscr(\overline\Omega)$
are in bijective correspondence with the topological 
isomorphism classes of principal $G$-bundles.
\end{theorem}

\begin{proof}
Choose a slightly larger Stein strongly pseudoconvex domain
$U$ containing $\overline\Omega$ as a deformation retract.
Extend a topological principal $G$-bundle to $U$ by pullback
under a retraction. The Oka--Grauert principle gives
a holomorphic principal $G$-bundle on $U$ in that topological
class. Its restriction has smooth, hence little, transition
maps. This proves surjectivity of the classification map.
For injectivity, the bundle of equivariant fibre isomorphisms
between two little-class principal bundles has fibre $G$ and
little-class holomorphic transition maps, given by left and
right multiplication. A topological isomorphism is a continuous
section. Since $G$ is Oka, Theorem \ref{th:main} supplies a
little-class holomorphic section, hence an isomorphism.
\end{proof}

\begin{corollary}
\label{cor:big-loss}
For every fixed $0<s<r$, a vector bundle of class
$\Lambda^r_\Oscr$ on $\overline\Omega$ has finitely many
generating sections and an embedding into a trivial bundle of
class $\lambda^s_\Oscr$. A subbundle of class $\Lambda^r_\Oscr$
has a complementary subbundle of class $\lambda^s_\Oscr$.
Topologically isomorphic vector bundles, or principal complex
Lie group bundles, of class $\Lambda^r_\Oscr$ are isomorphic
by an isomorphism of class $\lambda^s_\Oscr$. The corresponding
compact-parametric assertions follow in the same way.
The section and relative parametric conclusions are stated
explicitly in Theorems \ref{th:main-big} and
\ref{th:mainpar-big}.
\end{corollary}

\begin{proof}
Apply the continuous embedding $\Lambda^r\hookrightarrow\lambda^s$ 
in local trivialisations and use the little-space theorems at exponent
$s$, including Theorem \ref{th:OPprincipalbundles}.
\end{proof}

%
%
\begin{remark}\label{rem:Grassmann}
Grassmann manifolds $G(m,N)$ play a major role in the theory of 
complex vector bundles as classifying spaces. Being projective, 
they are K\"ahler manifolds. An explicit formula
for a K\"ahler metric on $G(m,N)$ was given by Lu Qi-Keng in 
1963; see \cite{Lu-Qi-Keng1963I,Lu-Qi-Keng1963II}.
\end{remark}

%
%
%
%
\medskip\noindent
{\bf Acknowledgements.} 
Research was supported by the European Union 
(ERC Advanced grant HPDR, 101053085) and 
grants P1-0291 and N1-0237 from ARIS, Republic of Slovenia.
I wish to thank Andrei Teleman for asking the question
which led to the main results of the paper (private communication,
January 2026), and for helpful communication regarding
the Zygmund spaces. I also thank Xianghong Gong
for sharing with me his expertise on Zygmund spaces.
A part of the background work on the paper originates in my 
joint papers \cite{DrinovecForstneric2007DMJ,DrinovecForstneric2008FM}
with Drinovec Drnov\v sek, whom I thank for this indirect contribution.
A part of the work was done during my visit to 
the University of Adelaide in February 2026. I wish to thank 
Finnur L{\'a}russon for the invitation and the University for its hospitality.
Finally, I thank the referees for their time and helpful comments.

%
%
%
%
\medskip\noindent\textbf{Declaration on the use of AI.} The initial version of this article was written without the use of AI. In a subsequent review, an AI tool helped identify a gap in the same-exponent approximation argument in the big H\"older--Zygmund spaces: smooth functions are not dense in the full $\Lambda^r$ norm. AI tools were subsequently used to assist with revising the little-space arguments, developing the $\Cscr^2$-boundary version, and editing the exposition. The author takes full responsibility for the final article.



\end{document}